\documentclass{amsart}
\usepackage{graphicx}
\usepackage{amssymb}
\usepackage{amsfonts}
\swapnumbers
\newtheorem{theorem}{Theorem}[section]
\newtheorem{lemma}[theorem]{Lemma}

\theoremstyle{definition}

\newtheorem{remark}[theorem]{Remark}

\numberwithin{equation}{section}
 \theoremstyle{plain}
 
 \numberwithin{equation}{section} %% Comment out for sequentially-numbered
 \numberwithin{figure}{section} %% Comment out for sequentially-numbered
 \theoremstyle{plain}
 \theoremstyle{remark}
 \newtheorem*{acknowledgement*}{Acknowledgement}
\newcommand{\cB}{{\mathcal B}}

\newcommand{\cF}{{\mathcal F}}
\newcommand{\cG}{{\mathcal G}}
\newcommand{\cH}{{\mathcal H}}
\newcommand{\cI}{{\mathcal I}}

\newcommand{\cL}{{\mathcal L}}

\newcommand{\cQ}{{\mathcal Q}}

\newcommand{\cW}{{\mathcal W}}

\newcommand{\fb}{{\mathfrak b}}
\newcommand{\Te}{{\Theta}}
\newcommand{\te}{{\theta}}
\newcommand{\vt}{{\vartheta}}
\newcommand{\Om}{{\Omega}}
\newcommand{\om}{{\omega}}
\newcommand{\ve}{{\varepsilon}}
\newcommand{\del}{{\delta}}
\newcommand{\Del}{{\Delta}}

\newcommand{\Gam}{{\Gamma}}

\newcommand{\vr}{{\varrho}}
\newcommand{\Sig}{{\Sigma}}
\newcommand{\sig}{{\sigma}}
\newcommand{\al}{{\alpha}}
\newcommand{\be}{{\beta}}
\newcommand{\ka}{{\kappa}}

\newcommand{\vs}{{\varsigma}}
\newcommand{\io}{{\iota}}
\newcommand{\bbR}{{\mathbb R}}

\newcommand{\bbZ}{{\mathbb Z}}
\newcommand{\bbI}{{\mathbb I}}
\begin{document}
\title[]{Strong diffusion approximation in averaging\\
with dynamical systems fast motions}%
 \vskip 0.1cm
 \author{ Yuri Kifer\\
\vskip 0.1cm
 Institute  of Mathematics\\
Hebrew University\\
Jerusalem, Israel}%
\address{
Institute of Mathematics, The Hebrew University, Jerusalem 91904, Israel}
\email{ kifer@math.huji.ac.il}%

\thanks{ }
\subjclass[2000]{Primary: 34C29 Secondary: 60F15, 60G40, 91A05}%
\keywords{averaging, diffusion approximation, $\phi$- and $\psi$-mixing,
 stationary process, shifts, dynamical systems.}%
\dedicatory{To Benji Weiss for his 80ies birthday  }
 \date{\today}
\begin{abstract}\noindent
The paper deals with the fast-slow motions setups in the continuous time
$\frac {dX^\ve(t)}{dt}=\frac 1\ve\Sig (X^\ve(t))\xi(t/\ve^2))+b(X^\ve(t),\,\xi(t/\ve^2),\, t\in [0,T]$ and the discrete time $X^\ve((n+1)\ve^2)=X^\ve(n\ve^2)+\ve\Sig(X^\ve(n\ve^2))\xi(n)
+\ve^2b(X^\ve(n\ve^2),\xi(n))$, $n=0,1,...,[T/\ve^2]$
 where $\Sig$ and $b$ are smooth matrix and vector functions and $\xi$ is a stationary vector stochastic process with weakly dependent terms and such that
 $E\xi(0)=0$. The assumptions imposed on the process $\xi$ allow applications to a wide
class of observables $g$ in the dynamical systems setup so that $\xi$ can be taken in the form $\xi(t)=g(F^t\xi(0))$ or $\xi(n)=g(F^n\xi(0))$ where $F$ is either a flow or a diffeomorphism with some hyperbolicity and $g$ is a vector function.
In this paper we show that both $X^\ve$ and a family of diffusions $\Xi^\ve$ can be redefined on a common sufficiently rich probability space
so that $E\sup_{0\leq t\leq T}|X^\ve(t)-\Xi^\ve(t)|^{p}\leq C\ve^\del,\, p\geq 1$ for some $C,\del>0$ and all $\ve>0$,
 where all $\Xi^\ve,\, \ve>0$ have the same diffusion coefficients but underlying Brownian motions may change with $\ve$.
\end{abstract}
%\footnotetext[1]{}
\maketitle
\markboth{Yu.Kifer}{Strong diffusion approximation}
\renewcommand{\theequation}{\arabic{section}.\arabic{equation}}
\pagenumbering{arabic}

\section{Introduction}\label{dynsec1}\setcounter{equation}{0}
Let $X^\ve$ be the solution of a system of ordinary differential equations having the form
\begin{equation}\label{1.1}
\frac {dX^\ve(t)}{dt}=\frac 1\ve\Sig(X^\ve(t))\xi(t/\ve^2)+b(X^\ve(t),\,\xi(t/\ve^2)),\,\, t\in[0,T]
\end{equation}
where $\Sig(x)$ and $b(x,\xi(s))$ are Lipschitz continuous matrix and vector fields
on $\bbR^d$ and $\xi$ is a stationary process which is viewed as a fast motion while $X^\ve$ is considered as a slow motion. Assume also that
for $s=0$ (and so for all $s$),
\begin{equation}\label{1.2}
E\xi(s)\equiv 0.
\end{equation}

It was shown in a series of papers \cite{Kha}, \cite{PK} and \cite{Bor} that $X^\ve$
(considered in a more general form)
converges weakly as $\ve\to 0$ to a diffusion process provided $\xi$ is sufficiently
fast mixing with respect to $\sig$-algebras generated by $\xi$ itself. It turns out
that the latter condition is quite restrictive when $\xi$ is generated by a dynamical system, i.e.\ when $\xi(t)=g\circ F^t$ where $g$ is a vector function and $F^t$ is a
flow (continuous time dynamical system) preserving certain measure which makes $\xi$
a stationary process. In order to derive weak convergence of $X^\ve$ to a diffusion
for $\xi$ generated by a sufficiently large class of dynamical systems other
approaches were developed recently based mainly on the rough paths theory (see
\cite{KM}, \cite{CFKMZ} and references there). All above mentioned results can
be obtained both in the continuous time setup (\ref{1.1}) and in the discrete time
setup given by the following recurrence relation
\begin{equation}\label{1.3}
X^\ve((n+1)\ve^2)=X^\ve(n\ve^2)+\ve \Sig(X^\ve(n\ve^2))\xi(n)+\ve^2b(X^\ve(n\ve^2),\,\xi(n))
\end{equation}
where $0\leq n<[T/\ve^2]$ and $\xi(n),\, n\geq 0$ is a stationary sequence of random
vectors. Observe that the above results can be viewed as a substantial generalization of the functional
central limit theorem since when $\Sig(x)$ does not depend on $x$ and $b\equiv 0$ the process $X^\ve$ weakly converges to the Brownian motion (with a covariance matrix).

Another, completely different, line of research dealt with extension of limit theorems for sums of random variables from convergence
in distribution or weak convergence to strong approximations or strong invariance principles results. This was done first for
independent random variables and martingales in \cite{Str}
 and extended to weakly dependent random variables in \cite{PS}, dealing in both cases
  with the one dimensional case since their proofs were based on the martingale
   approximation and the Skorokhod embedding theorem. The latter does not work,
    in general, in the multidimensional case (see \cite{MP2}) and another method
 was developed in \cite{BP} to tackle the case of sums $S_n=\xi_1+\cdots +\xi_n$
of weakly dependent random vectors. We observe that until recently (see \cite{MN2} where the general case
was treated) all papers dealing with strong approximations in the
multidimensional case starting with \cite{BP} (see \cite{DP} and references there)
 considered weak dependence or mixing with respect to $\sig$-algebras
generated by random vector summands $\xi(n),\, n\geq 0$ themselves which is quite
 restrictive in applications to dynamical systems. Only in the one dimensional case, which is based on the martingale approximation and the Skorokhod embedding, \cite{PS}
 considers a more general weak dependence setup which allowed to extend strong
 approximation theorems to dynamical systems in \cite{DeP} but only for one dimensional
 observables.

 Recently, strong $L^p$ diffusion approximations appeared in the first time for slow
 motions $X^\ve$ from (\ref{1.1}) and (\ref{1.3}) in \cite{Ki20}. The results there
 are valid in the multidimensional case but only assuming weak dependence of the
 processes $\xi(t),\, t\geq 0$ in (\ref{1.1}) and $\xi(n),\, n\geq 0$ in (\ref{1.3})
 with respect to the $\sig$-algebras generated by these random variables (vectors)
 themselves which, as mentioned above, allows applications to rather restricted class of
 observables in the dynamical systems setup. The purpose of this paper is to extend
 $L^p$ diffusion approximations results assuming more general weak dependence conditions
 which turns out to lead to substantial difficulties and will be achieved here only in
 certain situations. Such extension allows applications to a larger and natural class
 of vector observables in the dynamical systems setup. We stress that our
 goal is to obtain for each parameter value certain $L^p$ bounds on errors of diffusion
 approximations which is somewhat different from eventual almost sure bounds derived
 usually in the papers on strong approximations which dealt before only with sums of
 random variables (or vectors), and so the limiting process there was always the
 Brownian motion (with covariances).

\section{Preliminaries and main results}\label{dynsec2}\setcounter{equation}{0}
\subsection{Discrete time case}
We start with the discrete time setup which consists of a complete probability space
$(\Om,\cF,P)$, a stationary sequence of random vectors $\xi(n)$, $-\infty<n<\infty$ and
a two parameter family of countably generated $\sig$-algebras
$\cF_{m,n}\subset\cF,\,-\infty\leq m\leq n\leq\infty$ such that
$\cF_{mn}\subset\cF_{m'n'}\subset\cF$ if $m'\leq m\leq n
\leq n'$ where $\cF_{m\infty}=\cup_{n:\, n\geq m}\cF_{mn}$ and
$\cF_{-\infty n}=\cup_{m:\, m\leq n}\cF_{mn}$.
We will measure the dependence between $\sig$-algebras $\cG$ and $\cH$ by the $\phi$-coefficient
 defined by
\begin{eqnarray}\label{2.1}
&\phi(\cG,\cH)=\sup\{\vert\frac {P(\Gam\cap\Del)}{P(\Gam)}-P(\Del)\vert:\, P(\Gam)\ne
0,\,\Gam\in\cG,\,\Del\in\cH\}\\
&=\frac 12\sup\{\| E(g|\cG)-Eg\|_\infty:\, g\,\,\mbox{is $\cH$-measurable and }\,\|g\|_\infty=1\}\nonumber
\end{eqnarray}
(see \cite{Bra}) where $\|\cdot\|_\infty$ is the $L^\infty$-norm. For each $n\geq 0$ we set also
\begin{equation}\label{2.2}
\phi(n)=\sup_m\phi(\cF_{-\infty,m},\cF_{m+n,\infty}).
\end{equation}
If $\phi(n)\to 0$ as $n\to\infty$ then the probability measure $P$ is called $\phi$-mixing with respect to the family $\{\cF_{mn}\}$. Unlike \cite{Ki20}, in order
to ensure more applicability of our results to dynamical systems, we do not assume that
$\xi(n)$ is $\cF_{nn}$-measurable and instead we will work with the approximation
 coefficient
\begin{equation}\label{2.3}
\rho(n)=\sup_m\|\xi(m)-E(\xi(m)|\cF_{m-n,m+n})\|_\infty.
\end{equation}
To save notations we will still write $\cF_{mn}$, $\phi(n)$ and $\rho(n)$ for $\cF_{[m][n]}$, $\phi([n])$ and $\rho([n])$, respectively, if $m$ and $n$ are not integers (or $\pm\infty$), where $[\cdot]$ denotes the integral part.

We will deal with the recurrence relation (\ref{1.3}) where we set $X_N=X^{1/\sqrt N}$
so that (\ref{1.3}) takes the form
\begin{equation}\label{2.4}
X_N(n+1/N)=X_N(n/N)+\frac 1{\sqrt N}\Sig(X_N(n/N))\xi(n)+\frac 1Nb(X_N(n/N),\xi(n))
\end{equation}
and this definition is extended to all $t\in[0,T]$ by setting $X_N(t)=X_N(n/N)$ whenever
$n/N\leq t<(n+1)/N$.
We will assume that $\Sig(x)$ and $b(x,\cdot)$ are twice
and once differentiable in the first variable, respectively, and $b$ is Lipschitz continuous in the second variable. To avoid excessive technicalities these coefficients
are supposed to satisfy the following uniform bounds
\begin{equation}\label{2.5}
E\xi(0)=0,\,\|\xi(0)\|_\infty\leq L,\, \sup_{x\in\bbR^d}\max\big(| \Sig(x)|,\,|\nabla_x\Sig(x)|,| \nabla^2_x\Sig(x)|)\leq L,
\end{equation}
\begin{equation}\label{2.6}
\sup_{x,y\in\bbR^d}\max\big(| b(x,y)|,\,|\nabla_x b(x,y)|\big)\leq L\,\,\mbox{and}\,\,
 |b(x,y)-b(x,z)|\leq L|y-z|.
\end{equation}
for some constant $1\leq L<\infty$ and all $x,y,z\in\bbR^d$, where $\Sig=(\Sig_{ij},\,
 1\leq i,j\leq d)$ and $b=(b_1,...,b_d)$ are $d$-dimensional matrices and vectors and
 we take the Euclidean norms.

Set
\[
a_{jk}(x,m,n)=\sum_{i,l=1}^d\Sig_{ji}(x)\vs_{il}(n-m)\Sig_{lk}(x)
\]
where $\vs_{il}(n-m)=E(\xi_i(m)\xi_l(n))=E(\xi_i(0)\xi_l(n-m))$. We will see that under
the conditions of the theorem below there exist limits
\begin{equation}\label{2.13}
\vs_{ij}=\lim_{k\to\infty}\frac 1k\sum_{m=0}^k\sum_{n=0}^k\vs_{ij}(n-m),
\end{equation}
 where $\vs=(\vs_{ij})$, and so, for $j,k=1,...,d,$
\begin{equation}\label{2.9}
a_{jk}(x)=\lim_{n\to\infty}\frac 1n\sum_{k=0}^n\sum_{l=0}^na_{jk}(x,k,l)=
\sum_{i,l=1}^d\Sig_{ji}(x)\vs_{il}\Sig_{lk}(x).
\end{equation}

The matrix $A(x)=(a_{jk}(x))=\Sig(x)\vs\Sig^*(x)$ is twice differentiable
  if $\Sig(x)$ is. Since $\vs$ is symmetric and
nonnegatively definite, there exists $\vs^{1/2}$ such that $\vs^{1/2}\vs^{1/2}=\vs$.
Set $\sig(x)=\Sig(x)\vs^{1/2}$,
\[
c_i(x,m,n)=\sum_{j,k,l=1}^d\frac {\partial\Sig_{ij}(x)}{\partial x_k}\vs_{jl}(n-m)
\Sig_{kl}(x)
\]
and
\begin{equation}\label{2.14}
c_i(x)=\sum_{j,k,l=1}^d\frac {\partial\Sig_{ij}(x)}{\partial x_k}\hat\vs_{jl}
\Sig_{kl}(x),\, i=1,...,d
\end{equation}
where
\begin{equation}\label{2.15}
\hat\vs_{jl}=\lim_{k\to\infty}\frac 1k\sum_{n=0}^k\sum_{m=-k}^{n-1}\vs_{jl}(n-m)=
\sum_{m=1}^\infty E(\xi_j(m)\xi_l(0))
\end{equation}
and the latter limit will be shown to exist under our conditions.
It turns out that $\sig$, $b(x)=Eb(x,\xi(m))$ and $c(x)$ are Lipschitz continuous, and so there exists a unique solution $\Xi$ of the stochastic differential equation
\begin{equation}\label{2.11}
d\Xi(t)=\sig(\Xi(t))dW(t)+(b(\Xi(t))+c(\Xi(t)))dt
\end{equation}
where $W$ is the standard $d$-dimensional Brownian motion. When a non negatively definite symmetric matrix $A(x)$ is fixed then any solution of (\ref{2.11}) with any matrix $\sig$ satisfying $A(x)=\sig(x)\sig^*(x)$ has the same path distribution since this leads to the same Kolmogorov equation and to the same martingale problem (see \cite{SV}).

\begin{theorem}\label{thm2.1} Suppose that (\ref{1.2}) is satisfied, $X_N$ is defined by (\ref{2.4}) with $b$ and $\Sig$ satisfying (\ref{2.5}) and (\ref{2.6}).
 Assume that $\Sig^{-1}(x)=(\hat\Sig_{ij}(x))$ exists,
\begin{equation}\label{2.16}
\sup_x|\Sig^{-1}(x)|\leq L,
\end{equation}
\begin{equation}\label{2.17}
\frac {\partial\hat\Sig_{ij}(x)}{\partial x_k}=\frac {\partial\hat\Sig_{ik}(x)}{\partial x_j}\,\,\,\mbox{for all}\,\, i,j,k=1,...,d
\end{equation}
(which is automatically satisfied if $d=1$), and
\begin{equation}\label{2.18}
\phi(n),\,\rho(n)\leq C_0n^{-4}
\end{equation}
for some constant $C_0>0$ and all $n\geq 1$.
Then for each $N\geq 1$ the slow motion $X$ and the diffusion $\Xi$ having the same initial condition $X_N(0)=\Xi_N(0)$ can be redefined preserving their distributions on the same sufficiently rich probability space, which contains also an i.i.d.
sequence of uniformly distributed random variables, so that for any integers $N,M\geq 1$
and $T>0$,
\begin{equation}\label{2.20}
E\sup_{0\leq t\leq T}|X_N(t)-\Xi_N(t)|^{2M}\leq C(M,T)N^{-\del },
\end{equation}
 where $\del,\, C(M,T)>0$ do not depend on $N$  and
 they can be explicitly recovered from the proof. Here $\Xi=\Xi_N$ depends on $N$
 in the strong but not in
 the weak sense, i.e. the coefficients in (\ref{2.11}) do not depend on $N$ but
 for each $N$ in order to satisfy (\ref{2.20}) we may have to choose an appropriate Brownian motion $W=W_N$ which can be obtained from a
 universal Brownian motion $\cW$ by rescaling $W_N(t)=N^{-1/2}\cW(Nt)$.
 In particular, the Prokhorov distance between the distributions of $X_N$ and of $\Xi_N$ is bounded by $(C(M,T)N^{-\del })^{1/3}$.
\end{theorem}

The conditions of Theorem \ref{thm2.1} enable us in Section \ref{dynsec3} to reduce
the setup to the situation where $\Sig$ in (\ref{2.4}) is the identity matrix.
This allows to replace the process $\xi(n)$ by a process $\xi^{m_N}(n)=E(\xi(n)|\cF_{n-m_N,n+m_N})$ with an appropriate
$m_N\to\infty$ as $N\to\infty$ which, in turn, enables us to use the $\phi$-mixing effectively.
 After the replacement of $\xi(n)$ by $\xi^{m_N}(n)$ we rely on the strong approximation theorem
 and arrive at the required estimate comparing several auxiliary processes.

\begin{remark}\label{rem2.1+}
The reduction of the setup of Theorem \ref{thm2.1} to the situation where $\Sig$
is the the identity map allows
to obtain an almost sure version of the result, namely, to show that as $N\to\infty$ almost surely,
\[
\sup_{0\leq t\leq T}|X_N(t)-\Xi_N(t)|=O(N^{-\hat\del})
\]
for some $\hat\del>0$ independent of $N$.
%This is first obtained for a subsequence $N_K=K^{[2/\del]+1},\, K=1,2,...$
%using (\ref{2.20}), Chebyshev's inequality and the Borel-Cantelli lemma. %Secondly, we show that both $X_N$ and $\Xi_N$
%change negligibly when $N$ varies between $N_K$ and $N_{K+1}$.
\end{remark}

Important classes of processes satisfying our conditions come from
dynamical systems. Let $F$ be a $C^2$ Axiom A diffeomorphism (in
particular, Anosov) in a neighborhood of an attractor or let $F$ be
an expanding $C^2$ endomorphism of a Riemannian manifold $\Om$ (see
\cite{Bow}), $g$ be a either H\" older continuous vector functions or
vector functions
which are constant on elements of a Markov partition and let $\xi(n)=
\xi(n,\om)=g(F^n\om)$. Here the probability
space is $(\Om,\cB,P)$ where $P$ is a Gibbs invariant measure corresponding
to some H\"older continuous function and $\cB$ is the Borel $\sig$-field. Let
$\zeta$ be a finite Markov partition for $F$ then we can take $\cF_{kl}$
 to be the finite $\sig$-algebra generated by the partition $\cap_{i=k}^lF^i\zeta$.
 In fact, we can take here not only H\" older continuous $g$'s but also indicators
of sets from $\cF_{kl}$. The conditions of Theorem \ref{thm2.1} allow all such functions
since the dependence of H\" older continuous functions on $m$-tails, i.e. on events measurable
with respect to $\cF_{-\infty,-m}$ or $\cF_{m,\infty}$, decays exponentially fast in $m$ and
the condition (\ref{2.18}) is even weaker than that.  A related class of dynamical systems
corresponds to $F$ being a topologically mixing subshift of finite type which means that $F$
is the left shift on a subspace $\Om$ of the space of one (or two) sided
sequences $\om=(\om_i,\, i\geq 0), \om_i=1,...,l_0$ such that $\om\in\Om$
if $\pi_{\om_i\om_{i+1}}=1$ for all $i\geq 0$ where $\Pi=(\pi_{ij})$
is an $l_0\times l_0$ matrix with $0$ and $1$ entries and such that $\Pi^n$
for some $n$ is a matrix with positive entries. Again, we have to take in this
case $g$ to be H\" older continuous bounded functions on the sequence space above,
$P$ to be a Gibbs invariant measure corresponding
to some H\" older continuous function and to define
$\cF_{kl}$ as the finite $\sig$-algebra generated by cylinder sets
with fixed coordinates having numbers from $k$ to $l$. The
exponentially fast $\psi$-mixing, which is stronger than $\phi$-mixing
required here, is well known in the above cases (see \cite{Bow}). Among other
dynamical systems with exponentially fast $\psi$-mixing we can mention also the Gauss map
$Fx=\{1/x\}$ (where $\{\cdot\}$ denotes the fractional part) of the
unit interval with respect to the Gauss measure $G$ and more general transformations generated
by $f$-expansions (see \cite{Hei}). Gibbs-Markov maps which are known to be exponentially fast
$\phi$-mixing (see, for instance, \cite{MN}) can be taken as $F$ in Theorem \ref{thm2.1}(i)
with $\xi(n)=g\circ F^n$ as above.

\subsection{Continuous time case}

Here we start with a complete probability space $(\Om,\cF,P)$, a
$P$-preserving invertible transformation $\vt:\,\Om\to\Om$ and
a two parameter family of countably generated $\sig$-algebras
$\cF_{m,n}\subset\cF,\,-\infty\leq m\leq n\leq\infty$ such that
$\cF_{mn}\subset\cF_{m'n'}\subset\cF$ if $m'\leq m\leq n
\leq n'$ where $\cF_{m\infty}=\cup_{n:\, n\geq m}\cF_{mn}$ and
$\cF_{-\infty n}=\cup_{m:\, m\leq n}\cF_{mn}$. The setup includes
also a (roof or ceiling) function $\tau:\,\Om\to (0,\infty)$ such that
for some $\hat L>0$,
\begin{equation}\label{2.21}
\hat L^{-1}\leq\tau\leq\hat L.
\end{equation}
Next, we consider the probability space $(\hat\Om,\hat\cF,\hat P)$ such that $\hat\Om=\{\hat\om=
(\om,t):\,\om\in\Om,\, 0\leq t\leq\tau(\om)\},\, (\om,\tau(\om))=(\vt\om,0)\}$, $\hat\cF$ is the
restriction to $\hat\Om$ of $\cF\times\cB_{[0,\hat L]}$, where $\cB_{[0,\hat L]}$ is the Borel
$\sig$-algebra on $[0,\hat L]$ completed by the Lebesgue zero sets, and for any $\Gam\in\hat\cF$,
\[
\hat P(\Gam)=\bar\tau^{-1}\int_\Gam\bbI_\Gam(\om,t)dP(\om)dt\,\,\mbox{where}\,\,\bar\tau=\int\tau dP=E\tau
\]
and $E$ denotes the expectation on the space $(\Om,\cF,P)$.

Finally, we introduce a vector valued stochastic process $\xi(t)=\xi(t,(\om,s))$, $-\infty<t<\infty,\, 0\leq
s\leq\tau(\om)$ on $\hat\Om$ satisfying
\begin{eqnarray*}
&\xi(t,(\om,s))=\xi(t+s,(\om,0))=\xi(0,(\om,t+s))\,\,\mbox{if}\,\, 0\leq t+s<\tau(\om)\,\,\mbox{and}\\
&\xi(t,(\om,s))=\xi(0,(\vt^k\om,u))\,\,\mbox{if}\,\, t+s=u+\sum_{j=0}^k\tau(\vt^j\om)\,\,\mbox{and}\,\,
0\leq u<\tau(\vt^k\om).
\end{eqnarray*}
This construction is called in dynamical systems a suspension and it is a standard fact that $\xi$ is a
stationary process on the probability space $(\hat\Om,\hat\cF,\hat P)$ and in what follows we will write
also $\xi(t,\om)$ for $\xi(t,(\om,0))$.

We will assume that $X^\ve(t)=X^\ve(t,\om)$ from (\ref{1.1}) considered as a process on $(\Om,\cF,P)$
solves the equation
\begin{equation}\label{2.22}
\frac {dX^\ve(t)}{dt}=\frac 1\ve \Sig(X^\ve(t))\xi(t/\ve^2)+b(X^\ve(t),\,\xi(t/\ve^2)),\,\, t\in[0,T]
\end{equation}
where the matrix function $\Sig$ and the process $\xi$ satisfy (\ref{2.5}). Set
\begin{eqnarray}\label{2.23}
&\eta(\om)=\int_0^{\tau(\om)}\xi(s,\om)ds,\,\,\,\hat b(x,\om)=\int_0^{\tau(\om)}b(x,\xi(s,\om))ds\,\,\mbox{and}\\
&\rho(n)=\sup_m\max\big(\|\tau\circ\vt^m-E(\tau\circ\vt^m|\cF_{m-n,m+n})\|_\infty,\nonumber\\
&\|\eta\circ\vt^m-E(\eta\circ\vt^m|\cF_{m-n,m+n})\|_\infty,\nonumber\\
&\sup_x\|\hat b(x,\cdot)\circ\vt^m-E(\hat b(x,\cdot)\circ\vt^m|\cF_{m-n,m+n})\|_\infty\big).\nonumber
\end{eqnarray}
Since we will assume that $b(x,\zeta)$ is Lipschitz continuous in the first variable the last $\sup_x$ is
still measurable. Observe also that $\eta(k)=\eta\circ\vt^k$ and $\hat b(x,\cdot)\circ\vt^k$, $k\in\bbZ$ are
stationary sequences of random vectors.

Next, we consider a diffusion process $\Xi$ solving the stochastic differential equation (\ref{2.11}) with
$\sig^2(x)=A(x)=(a_{jk}(x))_{j,k=1,...,d}$, $b(x)=E\hat b(x,\cdot)$,
\begin{equation}\label{2.24}
a_{jk}(x)=\sum_{i,l=1}^d\Sig_{ji}(x)\vs_{il}\Sig_{lk}(x),\,\vs_{ij}=\lim_{n\to\infty}\frac 1n
\sum_{k,l=0}^nE(\eta_i(k)\eta_j(l))
\end{equation}
and
\begin{eqnarray}\label{2.25}
&c_{i}(x)=\sum_{j,k,l=1}^d\frac {\partial\Sig_{ij}(x)}{\partial x_k}\big(\hat\vs_{jl}+\frac 12E(\eta_j(0)\eta_l(0))\big)\Sig_{kl}(x),\\
&\hat\vs_{ij}=\lim_{n\to\infty}\frac 1n\sum_{k=0}^n\sum_{l=-n}^{k-1}E(\eta_i(k)\eta_j(l))=\sum_{m=1}^\infty E(\eta_i(m)\eta_j(0)).\nonumber
\end{eqnarray}
Notice the difference in the definitions of $c(x)$ in (\ref{2.14}) and in (\ref{2.25}) which is due to the fact that $c(x)$
is defined here in terms of the process $\eta$ and not $\xi$. The following result is a continuous time version of
Theorem \ref{thm2.1}(i).

\begin{theorem}\label{thm2.2} Assume that $E\eta=E\int_0^\tau\xi(t)dt=0$, $\xi$ and $\Sig$ satisfy the bound from (\ref{2.5})
and (\ref{2.16})--(\ref{2.18}) (with $\rho$ from (\ref{2.23}))  holds true, as well. Then the limits in (\ref{2.24}) and (\ref{2.25})
exist and for each $\ve>0$ the slow motion $X^\ve$ and the diffusion $\Xi=\Xi^\ve$ defined by (\ref{2.11}) which have the
same initial conditions $X^\ve(0)=\Xi^\ve(0)$ can be redefined preserving their (joint) distributions on the same sufficiently rich probability space, which contains also an i.i.d.
sequence of uniformly distributed random variables so that
for any integer $M\geq 1$ and numbers $\ve,T>0$,
\begin{equation}\label{2.26}
E\sup_{0\leq t\leq T}|X^\ve(t)-\Xi^\ve(t/\bar\tau)|^{2M}\leq C(M,T)\ve^\delta
\end{equation}
where $\del,\, C(M,T)>0$ do not depend on $\ve$. Here $\Xi=\Xi^\ve$ depends on $\ve$
 in the strong but not in the weak sense, i.e. the coefficients in (\ref{2.11}) do not depend on $\ve$ but
 for each $\ve$ in order to satisfy (\ref{2.26}) we may have to choose an appropriate Brownian motion $W=W_\ve$ which can be obtained from a universal Brownian motion $\cW$ by rescaling $W_\ve(t)=\ve\cW(\ve^{-2}t)$.
 In particular,
  the Prokhorov distance between the distributions of $X^\ve$ and of $\Xi^\ve$ is bounded by $(C(M,T)\ve^{\del })^{1/3}$.
\end{theorem}

The strategy of the proof of Theorem \ref{thm2.2} consists of three steps. First, we transform (\ref{2.22}) to another
equation where $\Sig$ is the identity matrix. Secondly, we consider a related discrete time setup which is treated by
means of Theorem \ref{thm2.1}. Thirdly, we see that the difference between this discrete and the original continuous
time processes is small.

The main application to dynamical systems we have here in mind is a $C^2$ Axiom A flow $F^t$ near an attractor which
using Markov partitions can be represented as a suspension over an exponentially fast $\psi$-mixing transformation so that
we can take $\xi(t)=g\circ F^t$ for a H\" older continuous function $g$ and the probability $P$ being a Gibbs invariant
measure constructed by a H\" older continuous potential on the base of the Markov partition (see, for instance, \cite{BR}). We observe that Prokhorov
distance estimates between the distributions of $X^\ve$ and $\Xi^\ve$ in the
one dimensional case for the (discrete time) suspension setup similar to
Theorem \ref{thm2.2} were obtained recently in \cite{AM}.

\section{Preliminary estimates  }\label{dynsec3}\setcounter{equation}{0}
\subsection{General lemmas}
First, we will formulate three general results which will be used throughout this
paper. The following lemma is well known (see, for instance, Corollary
to Lemma 2.1 in \cite{Kha} or Lemma 1.3.10 in \cite{HK}).
\begin{lemma}\label{lem3.1}
Let $H(x,\om)$ be a bounded measurable function on the space $(\bbR^d\times\Om,\,\cB\times\cF)$,
where $\cB$ is the Borel $\sig$-algebra, such that for each $x\in\bbR^d$ the function $H(x,\cdot)$
is measurable with respect to a $\sig$-algebra $\cG\subset\cF$. Let $V$ be an $\bbR^d$-valued
random vector measurable with respect to another $\sig$-algebra $\cH\subset\cF$.
Then with probability one,
\begin{equation}\label{3.1}
 |E(H(V,\om)|\cH)-h(V)|\leq 2\phi(\cG,\cH)\| H\|_{\infty}
 \end{equation}
where $h(x)=EH(x,\cdot)$ and the $\phi$-dependence coefficient was defined in (\ref{2.1}).
In particular (which is essentially an equivalent statement), let
 $H(x_1,x_2),\, x_i\in\bbR^{d_i},\, i=1,2$ be a bounded Borel function and $V_i$ be
 $\bbR^{d_i}$-valued $\cG_i$-measurable random vectors, $i=1,2$ where $\cG_1,\cG_2\subset\cF$ are
 sub $\sig$-algebras. Then with probability one,
 \begin{equation*}
 |E(H(V_1,V_2)|\cG_1)-h(V_1)|\leq 2\phi(\cG_1,\cG_2)\| H\|_{\infty}.
 \end{equation*}
 \end{lemma}

 We will employ several times the following general moment estimate which appeared as Lemma 3.2.5 in \cite{HK} for random variables and was extended to random vectors in
 Lemma 3.4 from \cite{Ki20}.
  \begin{lemma}\label{lem3.2}
  Let $(\Om,\cF,P)$ be a probability space,  $\cG_j,\, j\geq 1$ be a filtration of $\sig$-algebras and $\eta_j,\, j\geq 1$ be
  a sequence of random $d$-dimensional vectors such that $\eta_j$
  is $\cG_j$-measurable, $j=1,2,...$. Suppose that for some integer $M\geq 1$,
  \[
  A_{2M}=\sup_{i\geq 1}\sum_{j\geq i}\| E(\eta_j|\cG_i)\|_{2M}<\infty
  \]
  where $\|\eta\|_p=(E|\eta|^p)^{1/p}$ and $|\eta|$ is the Euclidean norm of a (random) vector $\eta$.
  Then for any integer $n\geq 1$,
  \begin{equation}\label{3.2}
  E|\sum_{j=1}^n\eta_j|^{2M}\leq 3(2M)!d^MA_{2M}^{2M}n^M.
  \end{equation}
  \end{lemma}

  In order to obtain uniform moment estimates required by Theorem \ref{thm2.1} we will need the following general estimate which appeared as Lemma 3.7 in \cite{Ki20}.
\begin{lemma}\label{lem3.3} Let $\eta_1,\eta_2,...,\eta_N$ be random $d$-dimensional vectors and
$\cH_1\subset\cH_2\subset...\subset\cH_N$ be a filtration of $\sig$-algebras such that $\eta_m$ is
$\cH_m$-measurable for each $m=1,2,...,N$. Assume also that $E|\eta_m|^q<\infty$ for some $q>1$
and each $m=1,...,N$. Set $S_m=\sum_{j=1}^m\eta_j$. Then
\begin{equation}\label{3.3}
E\max_{1\leq m\leq N}|S_m|^q\leq 2^{q-1}\big((\frac q{q-1})^qE|S_N|^q+E\max_{1\leq m\leq N-1}|\sum^N_{j=m+1}E(\eta_j|\cH_m)|^q\big).
\end{equation}
\end{lemma}

\subsection{Limits, transformations and approximations}
The following result deals with the coefficients $c(x)$ and $a_{jk}(x)$ introduced in
Section \ref{dynsec2} and establishes their properties.
\begin{lemma}\label{lem3.4} Under the conditions of Theorem \ref{thm2.1} the limits
 (\ref{2.13}) and (\ref{2.15}) exist and
\begin{equation}\label{3.4}
\hat\vs_{ij}+\hat\vs_{ji}=\vs_{ij}-E(\xi_i(0)\xi_j(0)).
\end{equation}
Moreover, uniformly in $m$ for all $m,n\geq 0$,
 \begin{equation*}
 |n\hat\vs_{ij}-\sum_{k=m}^{m+n}\sum_{l=m-n}^{k-1}\hat\vs_{ij}(k-l)|\leq 2L\sum_{k=0}^n\sum_{l=n+k}^\infty(L\phi([l/3])+\rho([l/3])+L\rho^2([l/3]))
\end{equation*}
and
 \begin{equation*}
 |n\vs_{jk}(x)-\sum_{i=m}^{m+n}\sum_{l=m}^{m+n}\vs_{jk}(l-i)|
 \leq 2L\sum_{i=0}^n\sum_{l=n+i}^\infty(L\phi([l/3])+\rho([l/3])+L\rho^2([l/3])).
\end{equation*}
  Finally, $c(x)$ and $b(x)=Eb(x,\xi(0))$ are once and $a_{jk}(x)$ is twice differentiable for $j,k=1,...,d$ and for all $x\in\bbR^d$,
 \begin{eqnarray*}
 &|b(x)|\leq L,\,|\nabla_xb(x)|\leq L,\,\max(|c(x)|,\,|a_{jk}(x)|)\leq\hat L\\
 &=2L^3\sum_{l=0}^\infty(L\phi(l)+\rho(l)+L\rho^2(l)),\,
 \max\big(|\nabla_xc(x)|,\,|\nabla_xa_{jk}(x)|,\,|\nabla^2_xa_{jk}(x)|\big)
 \leq 8\hat L
\end{eqnarray*}
where $L$ is the same as in (\ref{2.5}) and (\ref{2.6}).
\end{lemma}
\begin{proof}
By (\ref{2.3}), (\ref{2.5}), (\ref{2.6}) and Lemma \ref{lem3.1},
\begin{eqnarray*}
&|E(\xi_j(l)\xi_k(m))|\leq 2L\rho(|m-l|/3)\\
&+|E\big(E(\xi_j(l)|\cF_{l-\frac 13|m-l|,l+\frac 13|m-l|})E(\xi_k(m)|\cF_{m-\frac 13|m-l|,n+\frac 13|m-l|})\big)|\\
&\leq 2L(L\phi(|m-l|/3)+\rho(|m-l|/3)).
\end{eqnarray*}
By (\ref{1.2}), (\ref{2.3}) and (\ref{2.6}),
\[
|E(\xi(m)|\cF_{m-n,m+n})|\leq L\rho(n),
\]
and so by (\ref{2.3}), (\ref{2.5}), (\ref{2.6}) and Lemma \ref{lem3.1},
\[
|\vs_{jk}(m-l)|\leq L^2(2\rho(|m-l|/3)+2L\phi(|m-l|/3)+L^2\rho^2(|m-l|/3))
\]
and
\[
|\hat\vs_{ij}(m-l)|\leq L^2(2\rho(|m-l|/3)+2L\phi(|m-l|/3)+L^2\rho^2(|m-l|/3)).
\]
The existence of the limits (\ref{2.13}) and (\ref{2.15}) follow easily from here.
The bounds on derivatives of $a_{jk},\, b$ and $c$ follow from (\ref{2.5}), (\ref{2.6}),
(\ref{2.13}), (\ref{2.9}), (\ref{2.14}) and (\ref{2.15}).

Finally, we write
\begin{eqnarray*}
&\sum_{m,n=0}^kE(\xi_i(m)\xi_j(n))=\sum_{m=0}^k\sum_{n=-k}^{m-1}E(\xi_i(m)\xi_j(n))\\
&+\sum_{n=0}^k\sum_{m=-k}^{n-1}E(\xi_i(m)\xi_j(n))+\sum_{n=0}^kE(\xi_i(n)\xi_j(n))\\
&-\sum_{m=0}^k\sum_{n=-k}^{-1}E(\xi_i(m)\xi_j(n))-\sum_{n=0}^k\sum_{m=-k}^{-1}E(\xi_i(m)\xi_j(n))
\end{eqnarray*}
and (\ref{3.4}) follows since by (\ref{2.18}) and the estimates above
\[
\lim_{k\to\infty}\frac 1k\sum_{m=0}^k\sum_{n=-k}^{-1}E(\xi_i(m)\xi_j(n))=
\lim_{k\to\infty}\frac 1k\sum_{n=0}^k\sum_{m=-k}^{-1}E(\xi_i(m)\xi_j(n))=0,
\]
completing the proof of this lemma.
\end{proof}

Next, under the assumption of Theorem \ref{thm2.1} we will transform $X_N$ given by
(\ref{2.4}) into a more convenient form to deal with. Let $r:\bbR^d\to\bbR^d,\, r(x)=(r_1(x),...,r_d(x))$ be a map with its Jacobi matrix (differential) given by
\begin{equation}\label{3.5}
Dr(x)=\hat\Sig(x)=\Sig^{-1}(x)
\end{equation}
which exists in view of (\ref{2.17}) (see, for instance, Section 8.10 in \cite{Fle}).
Since $r$ is a local diffeomorphism by (\ref{2.16}), it follows from (\ref{2.5})
and the classical Hadamard theorem (see, for instance, Theorem 5.1.5 in \cite{Ber})
that $r$ is a diffeomorphism.

For $x,\zeta\in\bbR^d$ and $i=1,...,d$ set $q(x,\zeta)=(q_1(x,\zeta),...,q_d(x,\zeta))$ with
\begin{eqnarray}\label{3.6}
&q_i(x,\zeta)=\frac 12(Hr_i(x)\Sig(x)\zeta,\Sig(x)\zeta)=\frac 12\sum_{k,l=1}^d
\frac {\partial ^2r_i(x)}{\partial x_k\partial x_l}(\Sig(x)\zeta)_k(\Sig(x)\zeta)_l\\
&=\frac 12\sum_{k,l,m,n=1}^d\frac {\partial\hat\Sig_{ik}(x)}{\partial x_l}\Sig_{km}(x)\Sig_{ln}(x)\zeta_m\zeta_n\nonumber\\
&=- \frac 12\sum_{k,l,m,n=1}^d\hat\Sig_{ik}(x)\frac {\partial\Sig_{km}(x)}{\partial x_l}\Sig_{ln}(x)\zeta_m\zeta_n\nonumber
\end{eqnarray}
where $H$ denotes the Hessian and we used that
\begin{eqnarray}\label{3.7}
 &\sum_{k=1}^d\big(\frac {\partial\hat\Sig_{ik}(x)}{\partial x_l}\Sig_{km}(x)+
 \hat\Sig_{ik}(x)\frac {\partial\Sig_{km}(x)}{\partial x_l}\big)\\
 &=\frac {\partial}{\partial x_l}(\sum_{k=1}^d\hat\Sig_{ik}(x)\Sig_{km}(x))
 =\frac {\partial}{\partial x_l}(\del_{im})=0.\nonumber
 \end{eqnarray}

 Now introduce $Y_N(t),\, t\in[0,T]$ defined by the recurrence relation
 \begin{eqnarray}\label{3.8}
 &Y_N(\frac {n+1}N)=Y_N(n/N)+N^{-1/2}\xi(n)\\
 &+N^{-1}\big(\Sig^{-1}(r^{-1}(Y_N(n/N))
 b(r^{-1}(Y_N(n/N)),\xi(n))+q(r^{-1}(Y_N(n/N)),\xi(n))\big)\nonumber
 \end{eqnarray}
 with $Y_N(0)=r(X_N(0))$ and $Y_N(t)=Y_N(n/N)$ whenever $n/N\leq t<\frac {n+1}N$.
 \begin{lemma}\label{lem3.5}
 Set $Z_N(t)=r(X_N(t))$. Then
 \begin{equation}\label{3.9}
 \sup_{0\leq t\leq T}|Z_N(t)-Y_N(t)|\leq C_1N^{-1/2}
 \end{equation}
 where $C_1>0$ does not depend on $N\geq 1$.
 \end{lemma}
 \begin{proof} By (\ref{3.4}) and the Taylor formula
 \begin{eqnarray}\label{3.10}
 &Z_N(\frac {n+1}N)-Z_N(n/N)=Dr(X_N(n/N))(X_N(\frac {n+1}N)-X_N(n/N))\\
 &+N^{-1}q(X_N(n/N),\xi(n))+N^{-3/2}\eta_N(X_N(n/N),X_N(\frac {n+1}N),\xi(n))\nonumber\\
 &=N^{-1/2}\xi(n)+N^{-1}\big(\Sig^{-1}(r^{-1}(Z_N(n/N)))b(r^{-1}(Z_N(n/N)),\xi(n))\nonumber\\
 &+q(r^{-1}(Z_N(n/N)),\xi(n))\big)+N^{-3/2}\eta_N(X_N(n/N),X_N(\frac {n+1}N),\xi(n))\nonumber
 \end{eqnarray}
 where $\eta_N$ is uniformly bounded vector function in view of (\ref{2.5}) and (\ref{2.16}), i.e.
 \begin{equation}\label{3.11}
 \sup_{N\geq 1,\, x,y,\zeta\in\bbR^d}|\eta_N(x,y,\zeta)|\leq C_2
 \end{equation}
 for some $C_2>0$ which can be estimated explicitly using (\ref{2.5}) and (\ref{2.16}).

 Next observe that by (\ref{2.5}), (\ref{2.6}), (\ref{2.17}) and (\ref{3.5}) there
 exists a constant $L_2>0$ (which also can be estimated explicitly from the data of
 Section \ref{dynsec2}) such that for all $x,y,\zeta,\eta\in\bbR^d$ with
 $|\zeta|,|\eta|\leq L$,
 \begin{eqnarray}\label{3.12}
 &|\Sig^{-1}(r^{-1}(x))b(r^{-1}(x),\zeta)-\Sig^{-1}(r^{-1}(y))b(r^{-1}(y),\zeta)|
 \leq L_2|x-y|,\\
 &|\Sig^{-1}(r^{-1}(x))b(r^{-1}(x),\zeta)-\Sig^{-1}(r^{-1}(x))b(r^{-1}(x),\eta)|
 \leq L_2|\zeta-\eta|,\nonumber\\
 & |q(r^{-1}(x),\zeta)-q(r^{-1}(y),\zeta)|\leq L_2|x-y|,\,\,
 |q(r^{-1}(x),\zeta)-q(r^{-1}(x),\eta)|\nonumber\\
 &\leq L_2|\zeta-\eta|,\,\,|\Sig^{-1}(r^{-1}(x))b(r^{-1}(x),\zeta)|\leq L_2\,\,\,
 \mbox{and}\,\,\, |q(r^{-1}(x),\zeta)| \leq L_2.\nonumber
 \end{eqnarray}
 This together with (\ref{3.7}) and (\ref{3.9})--(\ref{3.11}) yields
 \begin{equation}\label{3.13}
 |Z_N(n/N)-Y_N(n/N)|\leq 2L_2N^{-1}\sum_{k=1}^{n-1}|Z_N(k/N)-Y_N(k/N)|+C_2N^{-1/2}.
 \end{equation}
 Finally, applying to (\ref{3.13}) the discrete (time) Gronwall inequality (see,
 for instance, \cite{Cla}) we obtain that
 \begin{equation}\label{3.14}
 |Z_N(n/N)-Y_N(n/N)|\leq C_2N^{-1/2}\exp(2L_2nN^{-1})
 \end{equation}
 and (\ref{3.9}) follows.
 \end{proof}

 \begin{lemma}\label{lem3.6}
 The process $\Psi(t)=r(\Xi(t))$ solves the stochastic differential equation
 \begin{equation}\label{3.15}
 d\Psi(t)=\vs^{1/2}dW(t)+\big(\Sig^{-1}(r^{-1}(\Psi(t))b(r^{-1}(\Psi(t)))+q^E(r^{-1}(\Psi(t)))\big)dt
 \end{equation}
 where $q^E(x)=Eq(x,\xi(0))$ and, recall, $b(x)=Eb(x,\xi(0))$.
 \end{lemma}
 \begin{proof}
 By the It\^ o formula (see, for instance, Section 7.3 in \cite{Ki19}),
 \begin{equation}\label{3.16}
 d\Psi(t)=Dr(\Xi(t))\sig(\Xi(t))dW(t)+\big(Dr(\Xi(t))(b(\Xi(t))+c(\Xi(t)))+\hat c(\Xi(t))\big)dt
 \end{equation}
 where $\hat c(x)=(\hat c_1(x),...,\hat c_d(x))$ and
 \[
 \hat c_i(x)=\frac 12\sum_{k,j,l=1}^d\sig_{kl}(x)\sig_{jl}(x)
 \frac {\partial\hat\Sig_{ik}(x)}{\partial x_j}.
 \]
 By (\ref{2.13}) and (\ref{3.6}),
 \begin{equation}\label{3.17}
 \hat c_i(x)=-\frac 12\sum_{k,j,l=1}^d\hat\Sig_{ik}(x)
 \frac {\partial\Sig_{km}(x)}{\partial x_j}\vs_{ml}\Sig_{jl}(x).
 \end{equation}
 Now set $u(x)=Dr(x)c(x)$. Then by (\ref{2.14}), (\ref{2.15}), (\ref{2.17}) and (\ref{3.7}),
 \begin{eqnarray*}
 &u_i(x)=\sum_{k,j,l,m=1}^d\hat\Sig_{ij}(x)
 \frac {\partial\Sig_{jk}(x)}{\partial x_l}\hat\vs_{km}\Sig_{lm}(x)\\
 &=-\sum_{k,j,l,m=1}^d\frac {\partial\hat\Sig_{ij}(x)}{\partial x_l}\Sig_{jk}(x)\hat\vs_{km}\Sig_{lm}(x)
 =-\sum_{k,j,l,m=1}^d\frac {\partial\hat\Sig_{il}(x)}{\partial x_j}\Sig_{jk}(x)\hat\vs_{km}\Sig_{lm}(x)\\
 &=\sum_{k,j,l,m=1}^d\hat\Sig_{il}(x)\Sig_{jk}(x)\hat\vs_{km}
 \frac {\partial\Sig_{lm}(x)}{\partial x_j}
 =\sum_{k,j,l,m=1}^d\hat\Sig_{ij}(x)
 \frac {\partial\Sig_{jk}(x)}{\partial x_l}\hat\vs_{mk}\Sig_{lm}(x).
 \end{eqnarray*}
 This together with (\ref{3.4}) yields
 \[
 u_i(x)=\sum_{k,j,l,m=1}^d\hat\Sig_{ij}(x)\frac {\partial\Sig_{jk}(x)}
 {\partial x_l}(\frac {\hat\vs_{km}+\hat\vs_{mk})}2)\Sig_{lm}(x)
 \]
 and by (\ref{3.6}) and (\ref{3.17}),
 \[
 u_i(x)+\hat c(x)=Eq^E_i(x).
 \]
 Finally, (\ref{3.15}) follows from here and (\ref{3.16}).
 \end{proof}

The transformation appearing in Lemmas \ref{lem3.5} and \ref{lem3.6} was employed previously
in \cite{GM} though full details were provided there only in the one dimensional case.
Lemmas \ref{lem3.5} and \ref{lem3.6} show that for the proof of Theorem \ref{thm2.1}
it suffices to estimate $E\sup_{0\leq t\leq T}|Y_N(t)-\Psi(t)|^{2M}$ which will yield the
estimate in (\ref{2.20}) in view of (\ref{2.16}) and (\ref{3.5}). This allows
to deal only with the case when $X_N$ is given by (\ref{2.4}) with $\Sig(x)$ being a constant
matrix.
In order to use the $\phi$-dependence coefficient effectively it will be convenient to consider the processes $Y_N^{(m)},\, m\geq 1$ defined for $n=0,1,...,[TN]-1$ by the recurrence relation
\begin{eqnarray}\label{3.18}
&Y^{(m)}_N(\frac {n+1}N)=Y^{(m)}_N(n/N)+N^{-1/2}\xi^{(m)}(n)\\
&+N^{-1}\big(\hat b(r^{-1}(Y_N^{(m)}(n/N)),\,\xi^{(m)}(n))+q(r^{-1}(Y_N^{(m)}(n/N)),
\,\xi^{(m)}(n))\big)\nonumber
\end{eqnarray}
where $Y^{(m)}_N=Y_N(0)$, $\xi^{(m)}(n)=E(\xi(n)|\cF_{n-m,n+m})$, $\hat b(x,\zeta)=
\Sig^{-1}(x)b(x,\zeta)$ and we set $Y_N^{(m)}(t)=Y_N^{(m)}(n/N)$ when $n/N\leq t<
\frac {n+1}N$.
\begin{lemma}\label{lem3.7} For all $m,N\geq 1$,
\begin{equation}\label{3.19}
\max_{0\leq n\leq [TN]}|Y_N(n/N)-Y^{(m)}_N(n/N)|\leq (1+2L_2)N^{1/2}\rho(m)e^{2L_2}.
\end{equation}
\end{lemma}
\begin{proof}
It follows by (\ref{3.12}) that
\begin{eqnarray*}
&|Y_N(n/N)-Y^{(m)}_N(n/N)|\leq N^{-1/2}\sum_{k=1}^{n-1}|\xi(k)-\xi^{(m)}(k)|\\
&+2L_2N^{-1}\sum_{k=1}^{n-1}|Y_N(k/N)-Y^{(m)}_N(k/N)|
+2L_2N^{-1}\sum_{k=1}^{n-1}|\xi(k)-\xi^{(m)}(k)|\\
&\leq N^{-1/2}n(1+2N^{-1/2}L_2)\rho(m)
+2L_2N^{-1}\sum_{k=1}^{n-1}|Y_N(k/N)-Y^{(m)}_N(k/N)|,
\end{eqnarray*}
and so by the discrete Gronwall inequality (see \cite{Cla}) the estimate (\ref{3.19}) follows.
\end{proof}

Next, set $m_N=[N^{(1-\ka)/2}]$ where $1/2<\ka<2/3$ with $0<\io<1$,
$n_k=n_k(N)=3km_N,\, k=0,1,...,[\frac {TN}{3m_N}]$, $Y_{N,k}=Y_N^{(m_N)}
(n_k/N)$, $\fb(x,\xi)=\hat b(r^{-1}x,\xi)+q(r^{-1}x,\xi)$,
\begin{eqnarray*}
&\al_{N,k}=\sum_{l=n_k}^{n_{k+1}-1}\xi^{(m_N)}(l),\,\,
\be_{N,k}(x)=\sum_{l=n_k}^{n_{k+1}-1}\fb(x,\xi^{(m_N)}(l))\\
&\mbox{and}\,\,\be_{N,k}=\be_{N,k}(Y_{N,k-1}).
\end{eqnarray*}
Introduce the process
\[
\check Y_N(n/N)=Y_N^{(m_N)}(0)+\sum_{l=0}^{[n/3m_N]}(N^{-1/2}\al_{N,l}+N^{-1}\be_{N,l}).
\]
\begin{lemma}\label{lem3.9} For all $N\geq n>k\geq 0$ and $T>0$,
\begin{equation}\label{3.24}
|Y_N^{(m_N)}(n/N)-Y^{(m_N)}_N(k/N)-\check Y_N(n/N)+\check Y_N(k/N)|\\
\leq 6(L+6L_2)(1+T)N^{-(\ka-\frac 12)}.
\end{equation}
\end{lemma}
\begin{proof}
First, we write
\begin{eqnarray*}
&Y_N^{(m_N)}(\frac {n_{i+1}}N)-Y_N^{(m_N)}(\frac {n_{i}}N)
=N^{-1/2}\sum_{l=n_i}^{n_{i+1}-1}\big(\xi^{(m_N)}(l)\\
&+N^{-1/2}\fb(Y_N^{(m_N)}(l/N),\,\xi^{(m_N)}(l))\big)
=N^{-1/2}\al_{N,i}+N^{-1}(\be_{N,i}+R_i^{(N)})\\
\end{eqnarray*}
where relying on (\ref{3.12}) we estimate $|\fb|$ and the Lipschitz constant of $\fb$
by $L_2$ and obtain
\begin{eqnarray*}
&|R^{(N)}_{i}|\leq \sum_{l=n_i}^{n_{i+1}-1}|\fb(Y_N^{(m_N)}(l/N),\,\xi^{(m_N)}(l))
-\fb(Y_{N,i-1},\,\xi^{(m_N)}(l))|\\
&\leq L_2\sum_{l=n_i}^{n_{i+1}-1}|Y_N^{(m_N)}(l/N)-Y_{N,i-1}|\\
&\leq 2N^{-1/2}L_2\sum_{l=n_i}^{n_{i+1}-1}\sum_{j=n_{i-1}}^{l-1}(j-n_{i-1})
\leq 84L_2N^{(1-3\ka/2)}.
\end{eqnarray*}
Now, summing in $i$ from $[\frac k{3m_N}]$ to $[\frac n{3m_N}]$ and taking into
account that
\begin{eqnarray*}
&|Y_N^{(m_N)}(n/N)-Y_N^{(m_N)}([\frac n{m_N}]m_N)|+|Y_N^{(m_N)}(k/N)-
Y_N^{(m_N)}([\frac k{m_N}]m_N)|\\
&\leq 6(L+2L_2)TN^{-\ka/2},
\end{eqnarray*}
we obtain (\ref{3.24}).
\end{proof}

\subsection{Moment and characteristic functions estimates}
 We will need next the following moment estimate.
 \begin{lemma}\label{lem3.11}
 For any $n,M\geq 1$,
 \begin{equation}\label{3.25}
 E|\sum_{k=0}^{n-1}\xi(k)|^{2M}\leq C_3(M)n^M
 \end{equation}
 where $C_3(M)>0$ can be recovered from the proof and it does not depend on $n$.
 \end{lemma}
 \begin{proof}
 First, we write
 \[
 |\sum_{k=0}^{n-1}\xi(k)|^{2M}\leq d^{2M-1}\sum_{i=1}^d
 |\sum_{k=0}^{n-1}\xi_i(k)|^{2M}.
 \]
 Set $\zeta_{mr}^{(i)}=E(\xi_i(m)|\cF_{m-r,m+r})$. Then by the martingale convergence
 theorem for the Doob martingale with probability one
 \begin{equation}\label{3.26}
 \xi_i(m)=\lim_{n\to\infty}\zeta^{(i)}_{n2^n}=\zeta_{m1}^{(i)}+
 \sum_{r=1}^\infty(\zeta_{m2^r}^{(i)}-\zeta_{m2^{r-1}}^{(i)}).
 \end{equation}
 By (\ref{2.3}) and (\ref{2.6}),
 \[
 \|\zeta^{(i)}_{m,2^r}-\xi(m)\|_\infty\leq 2L\rho(2^r),
 \]
 and so
  \[
 \|\zeta^{(i)}_{m,2^r}-\zeta^{(i)}_{m,2^{r-1}}\|_\infty\leq 2L(\rho(2^r)+\rho(2^{r-1}))
 \]
 implying that the series (\ref{3.26}) converges in $L^\infty$.

 Set
 \[
 S_n^{(i)}=\sum_{m=1}^n\xi_i(m),\, S^{(i)}_{n0}=\sum_{m=1}^n\zeta_{m1}^{(i)}\,\,
 \mbox{and}\,\, S^{(i)}_{nr}=\sum_{m=1}^n(\zeta^{(i)}_{m2^r}-\zeta^{(i)}_{m2^{r-1}}).
 \]
 Put $\cG_m^{(r)}=\cF_{-\infty,m+2^r}$ and observe that $\zeta^{(i)}_{m2^r}$ is $\cG_m^{(r)}$-measurable. By (\ref{2.1}) and (\ref{2.2}) for any $m\geq k+2^{r+1} \geq 2^{r+1}$,
 \[
 |E(\zeta^{(i)}_{m2^r}-\zeta^{(i)}_{m2^{r-1}}|\cG_k^{(r)})|\leq 2\phi(m-k-2^{r+1})
 \|\zeta^{(i)}_{m2^{r}}-\zeta^{(i)}_{m2^{r-1}}\|_\infty.
 \]
 For $k\leq m<k+2^{r+1}$ we just use the trivial estimate
 \[
 |E(\zeta^{(i)}_{m2^r}-\zeta^{(i)}_{m2^{r-1}}|\cG_k^{(r)})|\leq
 \|\zeta^{(i)}_{m2^r}-\zeta^{(i)}_{m2^{r-1}}\|_\infty.
 \]
 By (\ref{2.1}) and (\ref{2.2}) we have also that for $m>k+2$,
 \[
 |E(\zeta^{(i)}_m|\cG_k^{(0)})|\leq 2L\phi(m-k-2).
 \]
 Combining the above estimates and taking into account (\ref{2.5}) we obtain that
 \[
 A^{(0)}_{2M}=\sup_{k\geq 1}\sum_{m\geq k}\| E(\zeta^{(i)}_{m1}|\cG_k^{(0)})\|_{2M}\leq2L(1+\sum_{l=0}^\infty\phi(l))
 \]
 and for $r\geq 1$,
 \[
 A^{(r)}_{2M}=\sup_{k\geq 1}\sum_{m\geq k}\|E(\zeta^{(i)}_{m2^r}-\zeta^{(i)}_{m2^{r-1}}|\cG_k^{(r)})\|_{2M}\leq8\rho(2^{r-1})
 (2^r+\sum_{l=0}^\infty\phi(l))
 \]
 where $\|\cdot\|_p$ is the $L^p$-norm.

 Now, applying Lemma \ref{lem3.2} it follows that
 \[
 E(S^{(i)}_{mr})^{2M}\leq 3(2M)!(A^{(r)}_{2M})^{2M}n^M.
 \]
 Hence, by the Minkowski (triangle) inequality
 \[
 \| S^{(i)}_n\|_{2M}\leq\sum_{r=0}^\infty\| S^{(i)}_{nr}\|_{2M}\leq (3(2M)!)^{1/2M}
 \sqrt n\sum_{r=0}^\infty A^{(r)}_{2M}
 \]
 and rising both parts of this inequality to $2M$-th power we obtain (\ref{3.25}).
 \end{proof}

 Next, for each $n\geq 1$ and $x\in\bbR^d$ introduce the characteristic function
 \[
 f_n(x,w)=E\exp(i\langle w,\, n^{-1/2}\sum_{k=0}^{n-1}\xi(k)\rangle),\, w\in\bbR^d
 \]
 where $\langle\cdot,\cdot\rangle$ denotes the inner product. We will need the following
 estimate.
 \begin{lemma}\label{lem3.12}
 For any $n\geq 1$ and $x\in\bbR^d$,
 \begin{equation}\label{3.27}
 |f_n(x,w)-\exp(-\frac 12\langle\vs w,\, w\rangle)|\leq C_4n^{-\wp}
 \end{equation}
 for all $w\in\bbR^d$ with $|w|\leq n^{\wp/2}$ where the matrix $\vs$ is given
 by (\ref{2.13}) and we can take $\wp\leq\frac 1{20}$
 and a constant $C_4>0$ independent of $n$ can be recovered from the proof.
 \end{lemma}
 \begin{proof}
 The left hand side of (\ref{3.27}) does not exceed 2 and for $n<16$ we estimate it by
  $2(16)^\wp n^{-\wp}$ which is at least 2. So, in what follows, we will assume that $n\geq 16$. Since $\xi(k)$ is not supposed to be measurable with respect to $\cF_{kk}$ and we have to rely instead on approximation estimates (\ref{2.18})), it is not possible to reduce (\ref{3.27})
   directly to one of standard results such as Theorem 3.23 in \cite{DP}, and so we
  will provide a proof here which employs the standard block-gap technique.

  Set $\nu(n)=[n[n^{3/4}+n^{1/4}]^{-1}]$, $q_k(n)=k[n^{3/4}+n^{1/4}]$, $r_k(n)=
  q_{k-1}(n)+n^{3/4}$ for $k=1,2,...,\nu(n)$ with $q_0(n)=0$. Next, we introduce for
  $k=1,...,\nu(n)$,
  \begin{eqnarray*}
  &y_k=y_k(n)=\sum_{q_{k-1}(n)\leq l<r_k(n)}\xi(l),\, z_k=z_k(n)=\sum_{r_k(n)\leq
  l<q_k(n)}\xi(l)\\
  &\mbox{and}\,\, z_{\nu(n)+1}=\sum_{q_{\nu(n)-1}\leq l<n}\xi(l).
  \end{eqnarray*}
  Then by Lemma \ref{lem3.11},
   \begin{eqnarray}\label{3.28}
  &E|\sum_{1\leq k\leq\nu(n)+1}z_k|^2\leq 2\nu(n)\sum_{1\leq k\leq\nu(n)}E|z_k|^2+2E|z_{\nu(n)+1}|^2\\
  &\leq 2C_3(1)((\nu(n))^2n^{1/4}+n^{3/4})\leq 4C_3(1)n^{3/4}.\nonumber
  \end{eqnarray}
  This together with  the Cauchy-Schwarz inequality yields,
  \begin{eqnarray}\label{3.29}
  &|f_n(x,w)-E\exp(i\langle w,n^{-1/2}\sum_{1\leq k\leq\nu(n)}y_k\rangle)|\\
  &\leq E|\exp(i\langle w,n^{-1/2}\sum_{1\leq k\leq\nu(n)+1}z_k\rangle)-1|\leq n^{-1/2}E\langle w,\sum_{1\leq k\leq\nu(n)+1}z_k\rangle\nonumber\\
  &\leq n^{-1/2}|w|E|\sum_{1\leq k\leq\nu(n)+1}z_k|\leq 2\sqrt {C_3(1)}|w|n^{-1/8}\nonumber
  \end{eqnarray}
  where we use that for any real $a,b$,
  \[
  |e^{i(a+b)}-e^{ib}|=|e^{ia}-1|\leq |a|.
  \]
   We will obtain (\ref{3.27}) from (\ref{3.29}) by estimating
  \begin{equation}\label{3.30}
  |E\exp(i\sum_{1\leq k\leq\nu(n)}\eta_k)-\exp(-\frac 12\langle\vs w,w\rangle)|\leq I_1+I_2
  \end{equation}
  where
  \begin{eqnarray*}
  &\eta_k=\langle w,n^{-1/2}y_k\rangle,\,\,\,
  I_1=|E\exp(i\sum_{1\leq k\leq\nu(n)}\eta_k)-\prod_{1\leq k\leq\nu(n)}Ee^{i\eta_k}|\\
  &\mbox{and}\,\,\, I_2=|\prod_{1\leq k\leq\nu(n)}Ee^{i\eta_k}-\exp(-\frac 12\langle \vs w,w\rangle)|.
  \end{eqnarray*}

  First, we write
  \begin{eqnarray}\label{3.31}
  &I_1\leq\sum_{m=2}^{\nu(n)}\big(|\prod_{m+1\leq k\leq \nu(n)}Ee^{i\eta_k}|\\
  &\times|E\exp(i\sum_{1\leq k\leq m}\eta_k)- E\exp(i\sum_{1\leq k\leq m-1}
  \eta_k)Ee^{i\eta_m}|\big)\nonumber\\
  &\leq\sum_{m=2}^{\nu(n)}|E\exp(i\sum_{1\leq k\leq m}\eta_k)-E\exp(i\sum_{1\leq k\leq m-1}\eta_k)
  Ee^{i\eta_m}|\nonumber
  \end{eqnarray}
  where $\prod_{\nu(n)+1\leq k\leq \nu(n)}=1$. Next, using the approximation coefficient $\rho$ and
  the inequality $|e^{ia}-e^{ib}|\leq |a-b|$, valid for any real $a$ and $b$, we obtain
  \begin{equation}\label{3.32}
  |e^{i\eta_m}-\exp(iE(\eta_m|\cF_{q_{m-1}(n)-n^{1/4}/3,\infty}))|\leq n^{1/2}|w|\rho(n^{1/4}/3)
  \end{equation}
  and
  \begin{eqnarray}\label{3.33}
  &\big\vert\exp(i\sum_{1\leq k\leq m-1}\eta_k)-\exp(iE(\sum_{1\leq k\leq m-1}
  \eta_k|\cF_{-\infty,r_{m-1}(n)+n^{1/4}/3}))\big\vert\\
  &\leq n^{-1/2}|w|(m-1)\rho(n^{1/4}/3).\nonumber
  \end{eqnarray}
  Hence, by (\ref{3.32}), (\ref{3.33}) and Lemma \ref{lem3.1},
  \begin{eqnarray}\label{3.34}
  &\big\vert E\exp(i\sum_{1\leq k\leq m}\eta_k)-E\exp(i\sum_{1\leq k\leq m-1}
  \eta_k)Ee^{i\eta_m}\big\vert\\
  &\leq\big\vert E\big(\exp(i\sum_{1\leq k\leq m}\eta_k)-\exp\big(iE(\sum_{1\leq k\leq m-1}
  \eta_k|\cF_{-\infty,r_{m-1}(n)+n^{1/4}/3})\nonumber\\
  &+\exp(iE(\eta_m|\cF_{q_{m-1}(n)-n^{1/4}/3,\infty})\big)\big)\big\vert\nonumber\\
  &+\big\vert E\big(\exp\big(iE(\sum_{1\leq k\leq m-1}\eta_k|
  \cF_{-\infty,r_{m-1}(n)+n^{1/4}/3})\nonumber\\
  &+\exp(iE(\eta_m|\cF_{q_{m-1}(n)-n^{1/4}/3,\infty})\big)\big)\nonumber\\
  &-E\exp(iE(\sum_{1\leq k\leq m-1}\eta_k|\cF_{-\infty,r_{m-1}(n)+n^{1/4}/3}))\nonumber\\
  &\times E\exp(iE(\eta_m|\cF_{q_{m-1}(n)-n^{1/4}/3,\infty}))\big\vert\nonumber\\
  &+E\big\vert\exp(iE(\sum_{1\leq k\leq m}\eta_k)-\exp(iE(\sum_{1\leq k\leq m-1}\eta_k|\cF_{-\infty,r_{m-1}(n)+n^{1/4}/3}))\big\vert\nonumber\\
  &+E\big\vert e^{i\eta_m}-E\exp(iE(\eta_m|\cF_{q_{m-1}(n)-n^{1/4}/3,\infty}))\big\vert\nonumber\\
  &\leq\phi(n^{1/4}/3) +4n^{1/4}|w|m\rho(n^{1/4}/3).\nonumber
  \end{eqnarray}
  This together with (\ref{3.31}) yields that
  \begin{equation}\label{3.35}
  I_1\leq n^{1/4}(\phi(n^{1/4}/3)+4n^{1/2}|w|\rho(n^{1/4}/3)).
  \end{equation}

  In order to estimate $I_2$ we observe that
  \[
  |\prod_{1\leq j\leq l}a_j-\prod_{1\leq j\leq l}b_j|\leq\sum_{1\leq j\leq l}|a_j-b_j|
  \]
  whenever $0\leq |a_j|, |b_j|\leq 1,\, j=1,...,l$, and so
  \begin{eqnarray}\label{3.36}
  &I_2\leq\sum_{1\leq k\leq \nu(n)}|Ee^{i\eta_k}-\exp(-\frac 1{2\nu(n)}\langle \vs w,w\rangle)|\\
  &\leq\frac 12\sum_{1\leq k\leq \nu(n)}|E\eta_k^2-\frac 1{\nu(n)}\langle \vs w,w\rangle|\nonumber\\
  &+ \sum_{1\leq k\leq \nu(n)}E|\eta_k|^3+ \frac 1{4\nu(n)}|\langle\vs w,w\rangle|^2\nonumber
 \end{eqnarray}
 where we use (\ref{1.2}) and that for any real $a$,
 \[
 |e^{ia}-1-ia+\frac {a^2}2|\leq |a|^3\,\,\mbox{and}\,\, |e^{-a}-1+a|\leq a^2\,\,\mbox{if}\,\, a\geq 0.
 \]

 Now,
 \begin{eqnarray*}
 &E\eta_k^2=n^{-1}E(\sum_{j=1}^dw_j\sum_{l=q_{k-1}(n)}^{r_k(n)}\xi_j(l))^2\\
 &=n^{-1}\sum_{j,l=1}^d
 w_jw_l\sum_{i=q_{k-1}(n)}^{r_k(n)}\sum_{m=q_{k-1}(n)}^{r_k(n)}\vs_{jl}(m-i).
 \end{eqnarray*}
 Hence, by Lemma \ref{lem3.4},
 \begin{equation}\label{3.37}
 |E\eta_k^2-n^{-1/4}\langle\vs w,w\rangle|\leq 6Ld|w|^2n^{-1}\sum_{l=0}^{n^{3/4}}\sum^\infty_{m=n^{3/4}+l}
 (L\phi(m)+\rho(m)).
 \end{equation}
 By the estimate of $\vs_{jk}(m-l)$ in Lemma \ref{lem3.4} we have also
 \begin{eqnarray}\label{3.38}
 &|(\frac 1{\nu(n)}-n^{-1/4})\langle\vs w,w\rangle|\\
 &\leq 12Ld(L\sum_{l=0}^\infty\phi(l)+\sum_{l=0}^\infty\rho(l))|w|^2([\frac n{n^{3/4}+n^{1/4}}]^{-1}-n^{-1/4}).\nonumber
 \end{eqnarray}
 Since we assume that $n\geq 16$,
 \begin{eqnarray}\label{3.39}
 &[\frac n{n^{3/4}+n^{1/4}}]^{-1}-n^{-1/4}\leq (\frac n{n^{3/4}+n^{1/4}}-1)^{-1}-n^{-1/4}\\
 &=n^{-1/2}\frac {1+n^{-1/4}+n^{-1/2}}{1-n^{-1/4}-n^{-3/4}}\leq 8n^{-1/2}.\nonumber
 \end{eqnarray}
 By Lemma \ref{lem3.11}, H\" older inequality and the stationarity of the process $\xi$,
 \begin{equation}\label{3.40}
 E|\eta_k|^3\leq n^{-3/2}|w|^3\big(E(\sum_{l=q_{k-1}(n)}^{r_k(n)}\xi(l))^4\big)^{3/4}\leq C_3^{3/4}(2)n^{-3/8}|w|^3.
 \end{equation}
 Again, by the estimate of $\vs_{jk}(m-l)$ in Lemma \ref{lem3.4},
 \begin{equation}\label{3.41}
 \frac 1{\nu(n)}\langle\vs w,w\rangle\leq 64Ldn^{-1/4}|w|^2\sum_{l=0}^\infty(L\phi(l)+\rho(l)).
 \end{equation}
 Now, collecting (\ref{3.36})--(\ref{3.41}) we obtain that
 \begin{eqnarray}\label{3.42}
 &I_2\leq 3Ld|w|^2n^{-3/4}\sum_{l=0}^{n^{3/4}}\sum_{m=n^{3/4}+l}^\infty(L\phi(m)+\rho(m))\\
 &+112Ld|w|^2n^{-1/4}\sum_{l=0}^\infty(L\phi(l)+\rho(l))+C_3^{3/4}(2)n^{-1/8}|w|^3.
 \nonumber\end{eqnarray}
 Finally, (\ref{3.28}), (\ref{3.29}), (\ref{3.35}) and (\ref{3.42}) yield (\ref{3.27}) completing the proof.
 \end{proof}

Next, we split each time interval $[n_{k-1},n_k]$ into a block and a gap before
 it in the following way. Set $l_k=l_k(N)=n_{k-1}(N)+3[m^{1/4}_N]$,
 \begin{eqnarray*}
 &Q_{N,k}=\sum_{j=l_k}^{n_k-1}\xi^{(m^{1/4}_N)}(j),\,
 R^{(1)}_{N,k}=\sum_{j=l_{k}}^{n_k-1}(\xi^{(m_N)}(j)
 -\xi^{(m^{1/4}_N)}(j)),\\
 &R^{(2)}_{N,k}=\sum_{j=n_{k-1}}^{l_k-1}\xi^{(m_N)}(j)\,\,
 \mbox{and}\,\, Q_{N,k}(n)=\sum_{k=1}^{k_N(n)}Q_{N,k}
 \end{eqnarray*}
 where $k_N(t)=\max\{ k:\, n_k\leq t\}$. Then
 \begin{equation}\label{3.43}
 |\sum_{0\leq k<k_N(n)}\al_{N,k}-Q_{N,k}(n)|\leq |R^{(1)}_N(n)|+|R^{(2)}_N(n)|
 \end{equation}
 where $\al_{N,k}$ was defined before Lemma \ref{lem3.9} and
 \[
 R^{(i)}_N(n)=|\sum_{1\leq k\leq k_N(n)}R^{(i)}_{N,k}|,\,\, i=1,2.
 \]
 It turns out that the contributions of $R^{(1)}_N$ and $R^{(2)}_N$ are negligible for our purposes as the following lemma shows.

 \begin{lemma}\label{lem3.13}
 For all $N,M\geq 1$,
 \begin{equation}\label{3.44}
 E\max_{0\leq n\leq TN}(R^{(1)}_N(n)+R^{(2)}_N(n))^{2M}\leq C_5(M)N^{M(3+\ka)/4}
 \end{equation}
 where  $C_5(M)>0$ does not depend on $N$ and it can be recovered from the proof.
 \end{lemma}
 \begin{proof}
 Set $\eta^{(i)}_k=R^{(i)}_{N,k},\, i=1,2$ and $\cG_k=\cF_{-\infty,n_k}$.
 Without loss of generality assume that $N\geq 3^{\frac 8{3(1-\ka)}}$, and so $3m_N^{1/4}\leq m_N$.
 Then $\eta^{(1)}_k$ is $\cG_{k+1}$-measurable and $\eta^{(2)}_k$ is $\cG_{k}$-measurable.
 By (\ref{2.3}), (\ref{2.5}), (\ref{2.6}), (\ref{2.18}) and by Lemma \ref{lem3.1} we conclude that
 for $k\geq l+2$,
 \begin{eqnarray*}
 &|E(\eta^{(1)}_k|\cG_l)|\leq 2\phi(n_{k-1}-n_{k-2}-2m_N)\|\eta_k^{(1)}\|_\infty\\
 &\leq 12Lm_N\rho(m_N^{1/4})\phi(m_N)\leq 12LC_0^2N^{-5(1-\ka)/2}))
 \end{eqnarray*}
 and
 \begin{equation*}
 |E(\eta^{(2)}_k|\cG_l)|\leq 6Lm_N^{1/4}\phi(m_N)\leq 6LC_0N^{-17(1-\ka)/8}.
 \end{equation*}
 When $k=l-1,\, k=l$ or $k=l+1$ we will just use the trivial estimates
 \[
 |E(\eta^{(1)}_k|\cG_l)|\leq 6Lm_N\rho(m_N^{1/4})=6LC_0\,\,\mbox{and}\,\,|E(\eta^{(2)}_k|\cG_l)|\leq 3Lm^{1/4}_N
 =3LN^{(1-\ka)/8}.
 \]
 Hence,
 \begin{eqnarray*}
 &A^{(1)}_{2M}=\max_{1\leq l\leq k_N(TN)}\sum_{l\leq k\leq k_N(TN)}\| E(\eta^{(1)}_k|\cG_l)\|_{2M}\\
 &\leq 6LC_0(2C_0N^{-2+3\ka}+3)
 \end{eqnarray*}
 and
 \begin{eqnarray*}
 &A^{(2)}_{2M}=\max_{1\leq l\leq k_N(TN)}\sum_{l\leq k\leq k_N(TN)}\| E(\eta^{(2)}_k|\cG_l)\|_{2M}\\
 &\leq 6L(C_0N^{-(5-9\ka)/4}+N^{(1-\ka)/8}).
 \end{eqnarray*}
 By Lemmas \ref{lem3.2} and \ref{lem3.3} we obtain (\ref{3.44}) from here (cf. Lemma 4.1 in \cite{Ki19}).
 \end{proof}

 Next, we will need the following corollary of Lemma \ref{lem3.12}.
 \begin{lemma}\label{lem3.14}
 For any $N\geq 1$ and $k\leq TN^{(1+\ka)/2}$, with probability one,
 \begin{eqnarray}\label{3.45}
 &\quad\,\,\, |E\big(\exp(i\langle w,\,(n_k-l_k)^{-1/2}Q_{N,k}\rangle)|\cF_{-\infty,n_{k-1}+m^{1/4}_N}\big)
 -g(w)|\\
 &\leq C_4(n_k-l_k)^{-\wp}+C_0(LN^{-2(1-\ka)}+2N^{-(1-\ka)/2})\leq C_6(n_k-l_k)^{-\wp}\nonumber
 \end{eqnarray}
 for all $w\in\bbR^d$ with $|w|\leq(n_k-l_k)^{\wp/2}$ where $g_x(w)=\exp(-
 \frac 12\langle\vs w,w\rangle)$, $C_6=C_4+2C_0(1+3L)$ and $\wp=\frac 1{20}$.
 \end{lemma}
 \begin{proof}
 Set $F(w)=E\exp(i\langle w,\,(n_k-l_k)^{-1/2}Q_{N,k}\rangle)$.
 Then by Lemma \ref{lem3.1},
 \begin{eqnarray*}
 &|E\big(\exp(i\langle w,\,(n_k-l_k)^{-1/2}Q_{N,k}\rangle)|
 \cF_{-\infty,n_{k-1}+m^{1/4}_N}\big)-F(w)|\\
 &\leq 2\phi(m^{1/4}_N)\leq 2\phi(N^{(1-\ka)/8}).\nonumber
 \end{eqnarray*}
 Since $|e^{i(a+b)}-e^{ib}|\leq |a|$ we obtain from (\ref{2.3}) and (\ref{2.6}) that for all $x\in\bbR^d$,
 \[
 |F(w)-f_{n_k-l_k}(w)|\leq Lm_N^{1/2}\rho(m_N),
 \]
  where $f_n(w)$ is the same as in Lemma \ref{lem3.12},
 and (\ref{3.45}) follows from (\ref{2.18}) and Lemma \ref{lem3.12}.
 \end{proof}

\section{Strong approximations  }\label{dynsec4}\setcounter{equation}{0}
Our strong approximations will be based on the following result which appears as
 Theorem 4.6 of \cite{DP}. As usual, we will denote by $\sig\{\cdot\}$ a $\sig$-algebra generated by
random variables or vectors appearing inside the braces and we write $\cG\vee\cH$ for the minimal
$\sig$-algebra containing both $\sig$-algebras $\cG$ and $\cH$.
\begin{theorem}\label{thm4.1}
Let $\{ V_m,\, m\geq 1\}$ be a sequence of random vectors with values in $\bbR^d$ defined on some
probability space $(\Om,\cF,P)$ and such that $V_m$ is measurable with respect to $\cG_m$, $m=1,2,...$
where $\cG_m,\, m\geq 1$ is a filtration of sub-$\sig$-algebras of $\cF$.
Assume that the probability space is rich enough
so that there exists on it a sequence of uniformly distributed on $[0,1]$ independent random variables $U_m,\, m\geq 1$ independent of $\vee_{m\geq 0}\cG_m$. For each $m\geq 1$, let $G_m$
 be a probability distribution on $\bbR^d$ with the characteristic function
 \[
 g_m(w)=\int_{\bbR^d}\exp(i\langle w,x\rangle)G_m(dx),\,\, w\in\bbR^d.
 \]
 Suppose that for some non-negative numbers $\nu_m,\del_m$ and $K_m\geq 10^8d$,
 \begin{equation}\label{4.1}
 E\big\vert E(\exp(i\langle w,V_m\rangle)|\cG_{m-1})-g_m(w)\big\vert \leq\nu_m
 \end{equation}
 for all $w$ with $|w|\leq K_m$, and that
 \begin{equation}\label{4.2}
 G_m(\{ x:\, |x|\geq\frac 12K_m\})<\del_m.
 \end{equation}
 Then there exists a sequence $\{ W_m,\, m\geq 1\}$ of $\bbR^d$-valued independent random
 vectors defined on $(\Om,\cF,P)$ such that $W_m$ is $\sig\{ V_m,U_m\}$-measurable, $W_m$ is
 independent of $\sig\{ U_1,...,U_{m-1}\}\vee\cG_{m-1}$ (and so also of $W_1,...,W_{m-1})$ and
 \begin{equation}\label{4.3}
 P\{ |V_m-W_m|\geq\vr_m\}\leq\vr_m
 \end{equation}
 where $\vr_m=16K^{-1}_m\log K_m+2\nu_m^{1/2}K_m^d+2\del_m^{1/2}$.
 In particular, the Prokhorov distance between the distributions $\cL(V_m)$ and $\cL(W_m)$
 of $V_m$ and $W_m$, respectively, does not exceed $\vr_m$.
 \end{theorem}

 In order to apply this theorem we set
 $V_{m}=(n_{m}-l_{m})^{-1/2}Q_{N,m}$, $\cG_{m}=\sig\{ V_1,...,V_m\}\subset
 \cF_{-\infty,n_{m}+m^{1/4}_N}$ and $g_m=g$ defined in Lemma \ref{lem3.14}, so that
 $G_m=G$ is the mean zero $d$-dimensional Gaussian distribution with the covariance matrix $\vs$ and the characteristic function $g$.
 By Lemma \ref{lem3.14} for $|w|\leq K_m$,
\begin{equation}\label{4.4}
 E\big\vert E\big(\exp(i\langle w,V_{m}\rangle)| \cG_{m-1}\big)-g_{m}(w)\big\vert\\
 \leq C_6(n_{m}-l_{m})^{-\wp}
 \end{equation}
 where we take $K_{m}=(n_m-l_m)^{\wp/4d}\leq (n_m-l_m)^{\wp/2},\,\wp=\frac 1{20}$
 and recall that $n_m-l_m=3m_N-3[m_N^{1/4}]$.
 Theorem \ref{thm4.1} requires that $K_m\geq 10^8d$ which will hold true in our
 case if $N\geq N_0=N_0(\ka,\wp)=(10^{256d/\wp}d^{32d/\wp})^{1/(1-\ka)}$.

Next, let $\Psi$ be a mean zero Gaussian random variable with the
covariance matrix $\vs$. Then by estimates of Lemma \ref{lem3.4} and the
Chebyshev inequality,
\begin{eqnarray}\label{4.5}
& G(\{ y\in\bbR^d:\, |y|\geq\frac 12(n_m-l_m)^{\frac \wp{4d}}\})\\
&\leq P\{|\Psi| \geq\frac 12(n_m-l_m)^{\frac \wp{4d}}\}\leq 4L^2d
(n_m-l_m)^{-\frac \wp{2d}}\nonumber\\
&=4L^2d3^{-\wp/2d}n^{-\wp(1-\ka)/16d}(N^{3(1-\ka)/8}-1)^{-\wp/2d}.\nonumber
\end{eqnarray}

Now, Theorem \ref{thm4.1} provides us with independent random vectors $\{ W_m,\, m\geq 1\}$
having the mean zero Gaussian distribution with the covariance matrix $\vs$ and such that
 \begin{eqnarray}\label{4.6}
 &\vr_{m}=\vr_m(N)=4\frac \wp d(n_m-l_m)^{-\wp/4d}\log(n_m-l_m)+2C_6^{1/2}(n_m-l_m)^{-\wp/4}\\
 &+2L\sqrt d3^{-\wp/4d}N^{-\wp(1-\ka)/32d}(N^{3(1-\ka)/8}-1)^{-\wp/4d}
 \leq C_7N^{-\wp(1-\ka)/8d}\nonumber
 \end{eqnarray}
 where $C_7>0$ does not depend on $N\geq 1$.

 As a crucial corollary of Theorem \ref{thm4.1} we will obtain next a uniform $L^{2M}$-bound on the difference between the sums of $(n_{k}-l_{k})^{1/2}V_k$'s and of $(n_k-l_k)^{1/2}W_{k}$'s. Set
 \[
 I(n)=\sum_{k:\, n_k\leq n}(n_k-l_k)^{1/2}(V_k-W_k).
 \]
 \begin{lemma}\label{lem4.2}
 For any integers $N\geq N_0(\ka,\wp)$ and $M\geq 1$,
 \begin{equation}\label{4.7}
 E\max_{0\leq n\leq NT}|I(n)|^{2M}\leq C_8(M)N^{M-\frac \wp{20d}(1-\ka)}
 \end{equation}
 where $\wp=\frac 1{20}$, $0<\ka <\frac \io{4+\io}$ and $C_8(M)>0$ does not depend
 on $N$.
 \end{lemma}
 \begin{proof}
The proof of (\ref{4.7}) will rely on Lemmas \ref{lem3.2} and \ref{lem3.3}, and so we
 will have to estimate the conditional expectations appearing there taking into account that $V_k$ is $\cG_k\subset\cF_{-\infty,n_{k}+m_N^{1/4}}$-measurable and $W_{k}$ is
 $\cG_k\vee\sig\{ U_1,...,U_k\}$-measurable. Let $k>j\geq 1$.
 Since $W_k$ is independent of $\cG_{k-1}\vee\sig\{ U_1,...,U_{k-1}\}$ we obtain that
 \begin{equation}\label{4.8}
 E(W_k|\cG_j\vee\sig\{ U_1,...,U_j\})=EW_k=0.
 \end{equation}
 Next, since $V_k$ is independent of
 $\sig\{ U_1,...,U_j\}$ and the latter $\sig$-algebra is independent of $\cG_j$ we obtain that (see, for instance, \cite{Chu}, p. 323),
 \begin{equation}\label{4.9}
 E(V_k|\cG_j\vee\sig\{ U_1,...,U_j\})=E(E(V_k|\cG_{j\vee (k-2)})|\cG_j).
 \end{equation}
 By Lemma \ref{lem3.1},
 \begin{eqnarray}\label{4.10}
 &|E(V_k|\cG_{j\vee (k-2)})|\\
 &=(n_{k}-l_{k})^{-1/2}|\sum_{i=l_{k}}^{n_{k}-1}
 E\xi^{(m_N^{1/4})}(i)|\cF_{-\infty,n_{j}\vee n_{k-2}+m_N^{1/4}})|\nonumber\\
 &\leq L(n_{k}-l_{k})^{-1/2}\sum_{i=l_{k}}^{n_{k}-1}(\phi(i-n_{j}\vee n_{k-2}-2m_N^{1/4})+\rho(m_N^{1/4})).\nonumber
 \end{eqnarray}

 Now, in order to bound $A_{2M}$ from Lemma \ref{lem3.2} it remains to consider the
 case $k=n$, i.e. to estimate $\| V_k-W_k\|_{2M}$ and
 then to combine it with (\ref{4.8})--(\ref{4.10}). By the Cauchy--Schwarz inequality for any $n\geq 1$,
 \begin{eqnarray}\label{4.11}
 &E|V_k-W_k|^{2M}=E(|V_k-W_k|^{2M}\bbI_{|V_k-W_k|\leq\vr_k})\\
 &+E(|V_k-W_k|^{2M}\bbI_{|V_k-W_k|>\vr_k})\nonumber\\
 &\leq\vr^{2M}_k+(E|V_k-W_k|^{4M})^{1/2}(P\{|V_k-W_k|>\vr_k\}^{\frac {1}2}\nonumber\\
 &\leq\vr^{2M}_k+\vr^{\frac {1}2}_k2^{2M}((E|V_k|^{4M})^{1/2}+(E|W_k|^{4M})^{1/2}).\nonumber
 \end{eqnarray}
 By Lemmas \ref{lem3.1} and \ref{lem3.11},
\begin{equation}\label{4.12}
(E|V_k|^{4M})^{1/2}\leq\big(C_3(2M)+2\phi(m_N))L^{4M}(n_k-l_k)^{2M}\big)^{1/2}.
\end{equation}
Since $W_k$ is a mean zero $d$-dimensional Gaussian random vector with the covariance matrix $\vs$
we obtain that
\begin{equation}\label{4.13}
E|W_k|^{4M}=\leq|\vs^{1/2}|^{4M}((4M)!)^{d/2}.
\end{equation}
Finally, combining (\ref{4.8})--(\ref{4.13}) with Lemmas \ref{lem3.2} and \ref{lem3.3}
we derive (\ref{4.7}) completing the proof (cf. Lemma 4.4 in \cite{Ki20}).
\end{proof}

Next, let $W(t),\, t\geq 0$ be a standard $d$-dimensional Brownian motion.
 Then the sequences of random vectors $\tilde W_k^\ve =\vs^{1/2}(W(n_k)-W(l_k))$
and $(n_{k}-l_{k})^{1/2}W_k,\, k\leq k_N(TN)$ have the same distributions.
Hence, by Lemma A1 from \cite{BP} the sequences $\xi(k)$ and $W_k$, $k\geq 1$ can be redefined without changing their joint distributions on a richer probability space where there exists a
standard $d$-dimensional Brownian motion $W(t)$ such that the pairs $(V_k,W_k)$
and $(V_k,(n_{k}-l_{k})^{-1/2}\vs^{1/2}(W(n_k)-W(l_k)))$, $k\leq k_N(TN)$, constructed by means
of the redefined processes, have the same joint distributions. Thus we can assume from now on that
\[
W_k=(n_{k}-l_{k})^{-1/2}\vs^{1/2}(W(n_k)-W(l_k)))
\]
and that these $W_k$'s satisfy properties asserted in Theorem \ref{thm4.1}, so that Lemma
\ref{lem4.2} holds true for them, as well.

Next, using the Brownian motion $W(t),\, t\geq 0$ constructed above we consider the new Brownian motion $W_N(t)= N^{-1/2} W(tN)$ and introduce the diffusion process $\Psi_N(t),\, t\geq 0$ solving
 the stochastic differential equation (\ref{3.15}) which we write now with $W_N$ as,
\[
d\Psi_N(t)=\vs^{1/2}dW_N(t)+\fb(\Psi_N(t))dt,\,\, \Psi_N(0)=x_0
\]
where $\fb(x)=E\fb(x,\xi(0))$.
 Now, we introduce the auxiliary process $\hat\Psi_N$ with coefficients frozen at times $n_k/N,\, k\leq k_N(TN)$,
\begin{equation*}
\hat\Psi_N(t)=x_0+W_N(n_{k_N(tN)}/N))
+N^{-1}\sum_{1\leq k\leq k_N(tN)}\fb(\Psi_N(n_{k-2}/N))(n_k-n_{k-1})
\end{equation*}
where $n_{-1}=n_0=0$ and $k_N(s)$ was defined before Lemma \ref{lem3.13}, and
estimate its deviation from $\Psi_N$.

\begin{lemma}\label{lem4.3} For all integers $M,N\geq 1$,
\begin{equation}\label{4.14}
E\max_{0\leq k\leq k_N(TN)}|\Psi_N(n_k/N)-\hat\Psi_N(n_k/N)|^{2M}
\leq C_9(M)N^{-M(2+\ka)/2}
\end{equation}
where $C_9(M)>0$ does not depend on $N$.
\end{lemma}
\begin{proof}
It follows from (\ref{3.12}) that $\fb$ has the Lipschitz constant $2L_2$, and taking into
account also the Cauchy--Schwarz inequality we have that
\begin{eqnarray*}
&E\max_{0\leq k\leq k_N(TN)}|\Psi_N(n_k/N)-\hat\Psi_N(n_k/N)|^{2M}\\
&\leq 2^{2M}L_2^{2M}N^{-(2M-1)}\sum_{0\leq k\leq k_N(TN)}\int_{n_{k-1}/N}^{n_k/N}E|\Psi_n(t)-\Psi_N(n_{k-2}/N)|^{2M}dt.
\end{eqnarray*}
Now, for $n_{k-1}/N\leq t\leq n_{k}/N$,
\begin{eqnarray*}
&E|\Psi_N(t)-\Psi_N(n_{k-2}/N)|^{2M}\\
&\leq 2^{M-1}(|\vs^{1/2}|^{2M}E|W_N(t)-W_N(n_{k-2}/N)|^{2M}+
E|\int_{n_{k-2}/N}^{n_k/N}\fb(\Psi_N(s))ds|^{2M})\\
&\leq 2^{2M-1}(\frac {6m_N}N)^M(d^{2M}(2M)!|\vs^{1/2}|^{2M}+L_2^{2M}(\frac {6m_N}N)^M)
\end{eqnarray*}
and (\ref{4.14}) follows.
\end{proof}

Next, we define
\begin{equation*}
\hat Y_N(t)=x_0+\sum_{0\leq k< k_N(tN)}\big(N^{-1/2}\al_{N,k}
+N^{-1}\fb(Y_{N,k-1})(n_{k+1}-n_k)
\end{equation*}
where $\al_{N,k}$ is the same as in Lemma \ref{lem3.9}. In order to use the estimate of
Lemma \ref{lem3.9} we will need first to compare $\hat Y_N$ with the sum appearing there.

\begin{lemma}\label{lem4.4} For all $N\geq 1$,
\begin{equation}\label{4.15}
E\sup_{0\leq t\leq T}|\hat Y_N(t)-\breve Y_N(t)|^{2M}\leq C_{10}(M)N^{-\frac M2\min(\frac 12,\,5-7\ka)}
\end{equation}
where $\breve Y_N$ is the same as in Lemma \ref{lem3.9} and $C_{10}(M)>0$ does not depend on $N$.
\end{lemma}
\begin{proof} The left hand side of (\ref{4.15}) equals
$ N^{-2M}E\sup_{0\leq t\leq T}|J(t)|^{2M}$ where
\[
J(t)=\sum_{0\leq k<  k_N(tN)}(\fb(Y_{N,k-1})(n_{k+1}-n_k)-\be_{N,k}).
\]
Now, by (\ref{2.1}), (\ref{2.3}), (\ref{3.12}) and Lemma \ref{lem3.1} for any $k\geq l+1$,
\begin{eqnarray}\label{4.17}
&\big\vert E(\fb(Y_{N,k-1})(n_{k+1}-n_k)-\be_{N,k})|\cF_{-\infty,n_l+m_N})\big\vert\\
&=\big\vert E\big(\sum_{j=n_k}^{n_{k+1}}E(\fb(Y_{N,k-1})-\fb(Y_{N,k-1},\xi(j))+\fb(Y_{N,k-1},\xi(j))\nonumber\\
&-\fb(Y_{N,k-1},\xi^{(m_N)}(j))|\cF_{-\infty,n_{k-1}+m_N})|\cF_{-\infty,n_l+m_N}\big)\big\vert\nonumber\\
&\leq LN^{(1-\ka)/2}(4\phi(N^{(1-\ka)/2})+\rho(N^{(1-\ka)/2})).
\nonumber\end{eqnarray}
When $l=k+1$ then $\fb(Y_{N,k-1})(n_{k+1}-n_k)-\be_{N,k}$ is $\cF_{-\infty,n_l}$-measurable
and we estimate the left hand side of (\ref{4.17}) then and when $l=k$ just by $2LN^{(1-\ka)/2}$. Thus, relying on
 Lemmas \ref{lem3.2}, \ref{lem3.3} and \ref{lem3.11} we obtain
\begin{equation}\label{4.18}
E\sup_{0\leq t\leq T}|J(t)|^{2M}\leq\tilde C(M)N^{M(3-\ka)/2}(1+N^{M(1+\ka)}(\phi(N^{(1-\ka)/2})+\rho(N^{(1-\ka)/2}))^{2M}
\end{equation}
for some $\tilde C(M)>0$ independent of $N$.
\end{proof}

Next, denote
\begin{equation*}
\tilde\Psi_N(t)=x_0+\vs^{1/2}W_N(n_{k_N(tN)}/N)
+N^{-1}\sum_{0\leq k< k_N(tN)}\fb(Y_{N,k-1})(n_{k+1}-n_k).
\end{equation*}
Then
\begin{eqnarray}\label{4.19}
&\quad E\sup_{0\leq s\leq T}|\hat Y_N(s)-\hat\Psi^\ve(s)|^{2M}\leq 2^{2M-1}(E\max_{0\leq k<k_N(TN)}|\hat Y_N(n_k/N)\\
&-\tilde\Psi_N(n_k/N)|^{2M}+E\max_{0\leq k<k_N(TN)}|\tilde\Psi_N(n_k/N)-\hat\Psi_N(n_k/N)|^{2M})\nonumber.
\end{eqnarray}
By Lemmas \ref{lem3.3}, \ref{lem3.13} and \ref{lem4.2} for any $N\geq N_0(\ka,\wp)$,
\begin{eqnarray}\label{4.20}
&E\max_{0\leq k\leq n}|\hat Y_N(n_k/N)-\tilde\Psi_N(n_k/N)|^{2M}\\
&\leq 2^{2M-1}N^{-M}\big(E\sup_{0\leq t\leq T}|I(t)|^{2M}+E\max_{1\leq n\leq N}
(|R^{(1)}_N(n)|+|R^{(2)}_N(n)|)^{2M}\big)\nonumber\\
&\leq C_{11}(M)(N^{-\frac \wp{20d}(1-\ka)}+N^{-M(1-\ka)/4})\nonumber
\end{eqnarray}
where $C_{11}(M)>0$ does not depend on $n$ and $N$.

In order to estimate the second term in the right hand side of (\ref{4.19}) introduce the
$\sig$-algebras $\cQ_s=\cF_{-\infty,s+m_N}\vee\sig\{ W(u),\, 0\leq u\leq s\}$ and observe that
by our construction for each $k$ the increment $W(n_{k+1})-W(l_{k+1})$ is independent of $\cQ_{n_{k-1}}$.
 On the other hand, for any $k\geq n$ both $Y_{N,k}$ and $\Psi_N(n_k/N)$ are
 $\cQ_{n_k}$-measurable. Observe that
\begin{eqnarray}\label{4.22}
&E\max_{0\leq k\leq k_N(TN)}|\tilde\Psi_N(n_k/N)-\hat\Psi_N(n_k/N)|^{2M}\\
&\leq 2^{2M-1}(E\max_{0\leq k\leq k_N(TN)}|\cI_1(n_k)|^{2M}+E\max_{0\leq k\leq k_N(TN)}|\cI_2(n_k)|^{2M})
\nonumber\end{eqnarray}
where
\begin{equation*}
\cI_1(n_k)=N^{-1}\sum_{0\leq l\leq k-1}(\fb(Y_N(n_{l-1}/N))
-\fb(\Xi_N(n_{l-1}/N)))(n_{l+1}-n_l)
\end{equation*}
and
\begin{equation*}
\cI_2(n_k)=\sum_{0\leq j\leq k-1}\vs^{1/2}(W_N(l_{j+1}/N)-W_N(n_j/N)).
\end{equation*}
By the estimates of Lemma \ref{lem3.4} and the Lipschitz continuity of $\fb$ we have
\begin{equation}\label{4.23}
|\cI_2(n_k)|^{2M}\leq C_{13}(M)\frac {m_N}{NT}\sum_{0\leq l\leq k-1}|Y_N^{(m_N)}(n_{l-1}/N)-\Psi_N(n_{l-1}/N)|^{2M}
\end{equation}
for some $C_{13}>0$ which does not depend on $N$ or $k$. Since $\cI_2(n_k)$ can be viewed as a stochastic integral
we can rely on the corresponding martingale moment inequalities (see Section 1.7 in \cite{Mao}) which yields that
\begin{equation}\label{4.24}
E\max_{0\leq j\leq k}|\cI_2(n_j)|^{2M}\leq (\frac {2M}{2M-1})^{2M}E|\cI_2(n_k)|^{2M}\leq C_{14}N^{-3(1-\ka)/8}
\end{equation}
for some $C_{14}(M)>0$ independent of $N$.

Now denote
\[
G_k=E\max_{0\leq l\leq k}|Y_N^{(m_N)}(n_l/N)-\Psi_N(n_l/N)|^{2M}.
\]
Then we obtain from (\ref{4.14}), (\ref{4.15}) and (\ref{4.19})--(\ref{4.24}) that
for $n\leq k_N(TN)$ and $N\geq N_0(\ka,\wp)$,
\begin{eqnarray*}
&G_n\leq (2^{2M}C_{11}(M)+3^{2M}C_{14}(M))N^{-\frac \wp{20d}(1-\ka)}\\
&+2^{2M-1}(C_{12}(M)+C_{13}(M))\frac {m_N}{NT}\sum_{0\leq k\leq n-1}G_k.
\end{eqnarray*}
By the discrete (time) Gronwall inequality (see, for instance, \cite{Cla}),
\begin{eqnarray}\label{4.25}
&G_{k_N(TN)}\leq (2^{2M}C_{11}(M)\\
&+3^{2M}C_{14}(M))N^{-\frac \wp{20d}(1-\ka)}\exp((2^{2M}+3^{2M})(C_{12}(M)+C_{13}(M)))
\nonumber\end{eqnarray}
and Theorem \ref{thm2.1} follows taking into account (\ref{2.6}), (\ref{3.5}) and Lemmas \ref{lem3.5}--\ref{lem3.7}. (cf. Section 4.4 in \cite{Ki20}).  \qed

\section{Continuous time case  }\label{dynsec5}\setcounter{equation}{0}

We start with a diffeomorphism $r:\,\bbR^d\to\bbR^d$ satisfying (\ref{3.5}) and
set $Z^\ve(t)=r(X^\ve(t))$. Then
\begin{equation}\label{5.1}
\frac {dZ^\ve(t)}{dt}=Dr(X^\ve(t))\frac {dX^\ve(t)}{dt}=\frac 1\ve\xi(t/\ve^2)+\fb(Z^\ve(t),\xi(t/\ve^2))
\end{equation}
where $\fb(z,\zeta)=Dr(r^{-1}(z))b(r^{-1}(z),\zeta)$. Taking into account Lemma \ref{lem3.4} we see that
the limits in (\ref{2.24}) and (\ref{2.25}) exist and we can define the diffusion $\Xi$ by (\ref{2.11}).
In view of the definition of $c(x)$ in (\ref{2.25}) the computations of Lemma \ref{lem3.6} yield that the
process $\Psi(t)=r(\Xi(t))$ solves the stochastic differential equation
\begin{equation}\label{5.2}
d\Psi(t)=\vs^{1/2}dW(t)+\fb(\Psi(t))dt
\end{equation}
where $\fb(z)=E\fb(z,\eta(0))$. Hence, in order to prove (\ref{2.26}) it suffices to show that for each
$\ve>0$ both $Z^\ve$ and $\Psi=\Psi^\ve$ can be redefined on $(\Om,\cF,P)$ preserving their distributions
so that for any $\ve,T>0$ and an integer $M\geq 1$,
\begin{equation}\label{5.3}
E\sup_{0\leq t\leq T}|Z^\ve(t)-\Psi(t/\bar\tau)|^{2M}\leq C_1(M,T)\ve^{\del_1}
\end{equation}
for some $\del_1,\, C_1(M,T)>0$ independent of $\ve$.

Changing the time $s=t/\ve^2$ we can write
\begin{equation}\label{5.4}
Z^\ve(s\ve^2)=Z^\ve(s\ve^2,\om)=Z^\ve(0,\om)+\ve\int_0^s\xi(u,\om)du+\ve^2\int_0^s\fb(Z^\ve(u\ve^2,\om),
\xi(u,\om))du.
\end{equation}
Set $\Te_n(\om)=\sum_{j=0}^{n-1}\tau\circ\te^j(\om)$, $\Te_0(\om)=0$.
Introduce the discrete time processes $Y^{\ve}(m\ve^2,\om)=Y^{\ve}(m\ve^2,(\om,0))$ by the recurrence
relation $Y^{\ve}(0,\om)=Z^\ve(0,\om)=Z^\ve(0,(\om,0))$ and for $n\geq 0$,
\begin{eqnarray}\label{5.5}
&Y^\ve((n+1)\ve^2,\om)=Y^\ve(n\ve^2,\om)+\ve\int_{\Te_n(\om)}^{\Te_{n+1}(\om)}\xi(u,\om)du\\
&+\ve^2\int_{\Te_n(\om)}^{\Te_{n+1}(\om)}\fb(Y^\ve(n\ve^2,\om),\,\xi(u,\om))du\nonumber\\
&=Y^\ve(n\ve^2,\om)+\ve\eta\circ\vt^n(\om)+\ve^2\hat\fb(Y^\ve(n\ve^2,\om),\cdot)\circ\vt^n\nonumber
\end{eqnarray}
where $\hat\fb(z,\om)=Dr(r^{-1}(z))\hat b(r^{-1}(z),\om)$.

\begin{lemma}\label{lem5.1}
For all $\ve>0$, $\om\in\Om$ and an integer $N\geq 1$,
\begin{equation}\label{5.6}
\max_{0\leq n<N}\sup_{\Te_n(\om)\leq s<\Te_{n+1}(\om)}|Z^\ve(s\ve^2,\om)-Y^\ve(n\ve^2,\om)|\leq\ve(1+\ve)L\hat L\exp(L^3(L^2+1)\hat LN\ve^2).
\end{equation}
\end{lemma}
\begin{proof}
Set
\[
Q(n,\om)=\sup_{\Te_n(\om)\leq s<\Te_{n+1}(\om)}|Z^\ve(s\ve^2,\om)-Y^\ve(n\ve^2,\om)|.
\]
Then by (\ref{2.5}), (\ref{2.6}), (\ref{2.16}), (\ref{3.5}) and the definition of $\fb$,
\begin{equation}\label{5.7}
Q(n+1,\om)\leq L^3(L^2+1)\ve^2\sum_{k=0}^nQ(k,\om)\tau\circ\vt^k(\om)+\ve L\tau\circ\vt^n(\om)(1+\ve)
\end{equation}
since $|\fb(z,\om)-\fb(y,\om)|\leq L^3(L^2+1)|z-y|$. Then by (\ref{2.21}) and the discrete Gronwall
inequality we see that for all integers $n\geq 1$,
\[
Q(n,\om)\leq \ve(1+\ve)L\hat L\exp(L^3(L^2+1)\hat Ln\ve^2)
\]
and (\ref{5.6}) follows.
\end{proof}

Next, observe that we can apply Theorem \ref{thm2.1} to the process $Y^\ve$. Indeed, though we have here
the terms $\hat\fb(z,\cdot)\circ\vt^n$ which are slightly more general than $\hat\fb(z,\eta\circ\vt^n)$
appearing in the setup of Theorem \ref{thm2.1} but, in fact, we used there only the appropriate decay
of the approximation coefficient $\rho$ together with the Lipschitz continuity of coefficients in the
second variable given by (\ref{2.6}) which is replaced here by the coefficient $\rho$ in (\ref{2.23}) playing
the same role (with no need in additional Lipschitz continuity) and the proof proceeds in the same way.
Moreover, if $(\Om,\cF,P)$ is Lebesgue (or standard probability) space then we can always represent
$\hat\fb(z,\cdot)\circ\vt^n$ in the form $\hat\fb(z,\zeta\circ\vt^n)$, where $\zeta(n)=\zeta\circ\vt^n$
is a real valued stationary process, and considering the two component stationary process
 $(\eta\circ\vt^n,\,\zeta\circ\vt^n)_{n\in\bbZ}$ we come back to the setup of Theorem \ref{thm2.1}.
 Observe also that though in Theorem \ref{thm2.1} we considered only subsequences $\ve=\ve_N=1/N$ the
 assertion remains true for all small $\ve>0$. Indeed, let $(N_\ve+1)^{-1}\leq\ve^2<N_\ve^{-1}$. Then
 \begin{eqnarray*}
 &|Y^\ve(n\ve^2)-Y^{N_\ve^{-1/2}}(nN^{-1}_\ve)|\leq |\ve-N_\ve^{-1/2}|\sum_{k=0}^{n-1}\eta\circ\vt^k|\\
 &+\ve^2L^3(L^2+1)\sum_{k=0}^{n-1}|Y^\ve(k\ve^2)-Y^{N_\ve^{-1/2}}(kN^{-1}_\ve)|+L^2T\ve^{-2}|\ve^2-N^{-1}_\ve|.
 \end{eqnarray*}
 By the discrete Gronwall inequality
 \[
 \max_{0\leq n\leq T/\ve^2}|Y^\ve(n\ve^2)-Y^{N_\ve^{-1/2}}(nN^{-1}_\ve)|\leq\tilde C(\frac {\ve^3}{\sqrt {1-\ve^2}}
 \max_{0\leq n\leq T/\ve^2}|\sum_{k=0}^{n-1}\eta\circ\vt^k|+\frac {\ve^2}{1-\ve^2}).
 \]
 where $\tilde C>0$ does not depend on $\ve$. Applying Lemmas \ref{lem3.1}--\ref{lem3.3} to the right hand side here
 (similarly to Lemmas \ref{lem3.11}, \ref{lem4.4} and Lemma \ref{lem5.2} below) we obtain that
 \[
 E\max_{0\leq n\leq T/\ve^2}|Y^\ve(n\ve^2)-Y^{N_\ve^{-1/2}}(nN^{-1}_\ve)|^{2M}\leq\tilde {\tilde C}
 \frac {\ve^{2M}}{(1-\ve^2)^M},
 \]
 for some $\tilde C>0$ independent of $\ve$,
 showing that we can apply Theorem \ref{thm2.1} for all $\ve>0$ when $X^\ve$ is in the form (\ref{1.3}) and not only when
 the parameters restricted to the sequence $\ve_N=1/\sqrt N$ provided, of course, other conditions of this assertion are met.

 Applying Theorem \ref{thm2.1} we obtain that both $Y^\ve$ and the diffusion $\Psi=\Psi^\ve$ can be redefined
 on $(\Om,\cF,P)$ without changing their distributions so that for any $\ve>0$ and an integer $M\geq 1$,
 \begin{equation}\label{5.8}
 \sup_{0\leq t\leq T/\bar\tau}|Y^\ve(t)-\Psi^\ve(t)|^{2M}\leq C_2(M,T)\ve^{\del_2}
 \end{equation}
 for some $\del_2,\, C_2(M,T)>0$ independent of $\ve$ where we set $Y^\ve(t)=Y^\ve(n\ve^2)$
 if $n\ve^2\leq t<(n+1)\ve^2$.

Now we have
\begin{eqnarray}\label{5.9}
& \sup_{0\leq t\leq T}|Z^\ve(t)-\Psi^\ve(t/\bar\tau)|^{2M}\leq 3^{2M-1}(E\sup_{0\leq t\leq T}|I_1^\ve(t)|^{2M}\\
&+E\sup_{0\leq t\leq T} |I^\ve_2(t)|^{2M}+E\sup_{0\leq t\leq T}|I_3^\ve(t)|^{2M})\nonumber
 \end{eqnarray}
 where
 \[
 I_1^\ve(t)=Y^\ve(t/\bar\tau)-\Psi^\ve(t/\bar\tau),\, I_2^\ve(t)=\sup_{\Te_{[t/\ve^2\bar\tau]}(\om)
 \leq s<\Te_{[t/\ve^2\bar\tau]}(\om)+1}|Y^\ve(t/\bar\tau,\om)-Z^\ve(s\ve^2,\om)|
 \]
 and
 \begin{equation}\label{5.10}
 I^\ve_3(t)=\sup_{\Te_{[t/\ve^2\bar\tau]}(\om) \leq s<\Te_{[t/\ve^2\bar\tau]}(\om)+1}|Z^\ve(t,\om)-Z^\ve(s\ve^2,\om)|\leq\ve
 L\bar\tau(1+\ve L)+\sup_{0\leq t\leq T}|I_4^\ve(t)|
 \end{equation}
 where
 \[
 I_4^\ve(t)=Z^\ve(t,\om)-Z^\ve(\ve^2\Te_{[t/\ve^2\bar\tau]}(\om),\om).
 \]
 The first and the second terms of (\ref{5.9}) are estimated by (\ref{5.8}) and (\ref{5.7}), respectively, and so in view of
 (\ref{5.10}) in order to derive (\ref{5.3}) it remains to prove the following result.

 \begin{lemma}\label{lem5.2}
 For $\ve>0$ and an integer $M\geq 1$,
 \begin{equation}\label{5.11}
 E\sup_{0\leq t\leq T}|I^\ve_4(t)|^{2M}\leq C_3(M,T)\ve^{\del_3}
 \end{equation}
 for some $\del_3,\, C_3(M,T)>0$ which do not depend on $\ve$.
 \end{lemma}
 \begin{proof} Set
 \[
 n(t,\om)=\max\{ k:\,\Te_k(\om)\leq t\},\,\,\Gam(t,q)=\{\om:\, |n(t,\om)-[t/\bar\tau]|>q\}
 \]
 and $\Del(k,q)=\{\om:\, |\Te_k(\om)-\bar\tau k|>q\}$. Then by (\ref{2.5}) and (\ref{2.21}),
 \begin{equation}\label{5.12}
 |I^\ve_4(t)|\leq\ve L\hat L(1+J_1^\ve(t)+\bar\tau)+\ve J^\ve_2([t/\ve^2\bar\tau])+\ve^2L^2J_3(t)
 \end{equation}
 where
 \begin{equation}\label{5.13}
 J_1^\ve(t)=\bbI_{\Gam(t/\ve^2,\ve^{-3/2})}|\Te_{[t/\ve^2\bar\tau]}-t/\ve^2|\leq\bbI_{\Del([t/\ve^2\bar\tau],
 \ve^{-3/2}-\hat L-2\hat L^{-1})}|\Te_{[t/\ve^2\bar\tau]}-\bar\tau[t/\ve^2\bar\tau]|,
 \end{equation}
 \begin{equation}\label{5.14}
 J_2(m)=\max_{0\leq n\leq\ve^{-3/2}}|\sum_{k:\, |k-m|\leq n}\eta\circ\vt^k|
 \end{equation}
 and
 \begin{equation}\label{5.15}
 J_3^\ve(t)=|\Te_{[t/\ve^2\bar\tau]}-t/\ve^2|\leq |\Te_{[t/\ve^2\bar\tau]}-\bar\tau[t/\ve^2\bar\tau]|+\bar\tau.
 \end{equation}

 By (\ref{2.5}), (\ref{2.21}), (\ref{2.23}) and Lemma \ref{lem3.1} for any $l\geq 1$,
 \begin{eqnarray}\label{5.16}
 &|E(\tau\circ\vt^{k+l}-\bar\tau|\cF_{-\infty,k})|\\
 &\leq |E\big(\tau\circ\vt^{k+l}-E(\tau\circ\vt^{k+l}|\cF_{k+l-[l/3],k+l+[l/3]})|\cF_{-\infty,k}\big)|\nonumber\\
 &+|E\big(E(\tau\circ\vt^{k+l}|\cF_{k+l-[l/3],k+l+[l/3]})-\bar\tau|\cF_{-\infty,k}\big)|\leq\rho([l/3])+2\hat L\phi([l/3])
 \nonumber\end{eqnarray}
 and similarly,
 \[
 |E(\eta\circ\vt^{k+l}|\cF_{-\infty,k}\big)|\leq\rho([l/3])+2L\bar\tau\phi([l/3]).
 \]
 This together with (\ref{2.18}) and Lemmas \ref{lem3.2} and \ref{lem3.3} yields that for any $n,N\geq 1$,
 \begin{equation}\label{5.17}
 E\max_{0\leq k\leq n}|\Te_k-k\bar\tau|^{2N}\leq C_4(N)n^N
 \end{equation}
 and
  \begin{equation}\label{5.18}
 E\max_{0\leq k\leq n}|\sum_{j=0}^n\eta\circ\vt^j|^{2N}\leq C_4(N)n^N
 \end{equation}
 for some $C_4(N)>0$ independent of $n$.

 Hence, for any $N\geq 1$,
 \begin{equation}\label{5.19}
 E\sup_{0\leq t\leq T}|J^\ve_3(t)|^{2N}\leq 2^{2N-1}(C_4(N)[T/\ve^2\bar\tau]^N+\bar\tau^N)
 \end{equation}
 and by the stationarity of the sequence $\eta\circ\vt^k,\, k\in\bbZ$,
 \begin{eqnarray}\label{5.20}
 & E\sup_{0\leq t\leq T}(J^\ve_2([t/\ve^2\bar\tau]))^{2N}\leq E\max_{0\leq m\leq T/\ve^2\bar\tau}(J_2^\ve(m))^{2N}\\
 & \leq\frac T{\ve^2\bar\tau}E(J_2^\ve(0))^{2N}\leq C_4(N)2^NT\bar\tau^{-1}\ve^{-(2+\frac 32N)}\nonumber
\end{eqnarray}
and by the Cauchy--Schwarz inequality,
\begin{eqnarray*}
&E\sup_{0\leq t\leq T}(J^\ve_3(t))^{2N}\leq E\max_{0\leq n\leq T/\ve^2\bar\tau}(
\bbI_{\Del(n,\ve^{-3/2}-\hat L-2\hat L^{-1})}|\Te_n-\bar\tau n|^{2N})\\
 &\leq\sum_{0\leq n\leq T/\ve^2\bar\tau}E(\bbI_{\Del(n,\ve^{-3/2}-\hat L-2\hat L^{-1})}|\Te_n-\bar\tau n|^{2N})\\
 &\leq\sum_{0\leq n\leq T/\ve^2\bar\tau}\big(P(\Del(n,\ve^{-3/2}-\hat L-2\hat L^{-1}))E|\Te_n-\bar\tau n|^{4N}
 \big)^{1/2}\end{eqnarray*}

 By (\ref{5.17}) and the Chebyshev inequality for any $K\geq 1$,
 \[
 P(\Del(n,\ve^{-3/2}-\hat L-2\hat L^{-1}))\leq (\ve^{-3/2}-\hat L-2\hat L^{-1})^{-2K}(T/\bar\tau)^K\ve^{-2K},
 \]
 and so
 \begin{equation}\label{5.21}
 E\sup_{0\leq t\leq T}(J_1^\ve(t))^{2N}\leq C_5(K,N,T)\ve^{\frac 12K-2N-2}
 \end{equation}
 for some $C_5(K,N,T)>0$ independent of $\ve>0$. Finally, we obtain (\ref{5.11}) from (\ref{5.12})--(\ref{5.15}) and
 (\ref{5.19})--(\ref{5.21}) taking $N=M$ and $K>4(M-1)$, completing the proof of both Lemma \ref{lem5.2} and
 Theorem \ref{thm2.2}.
  \end{proof}

%\bibliography{matz_nonarticles,matz_articles}
%\bibliographystyle{alpha}

\end{document}